\documentclass[11pt,a4paper, reqno]{amsart}
\usepackage{mathrsfs}

\usepackage{amsmath,amsfonts,verbatim}
\usepackage{latexsym}
\usepackage{amssymb,leftidx}
\usepackage{extarrows}
\usepackage{overpic}
\usepackage{color}
\usepackage{epsfig}
\usepackage{subfigure}
\usepackage{tikz}
\usepackage{hyperref}
\colorlet{RED}{red}

\hypersetup{urlcolor=blue, citecolor=red}

\usepackage{amssymb}
\usepackage{mathrsfs}
\usepackage{amscd}
\usepackage{bbm}
\usepackage{cite}

\newcommand\R{\mathbb{R}}

\newcommand\N{\mathbb{N}}

\usepackage{graphicx}
\usepackage{float}

\numberwithin{equation}{section}
\newtheorem{proposition}{Proposition}[section]

\newtheorem{lemma}{Lemma}[section]
\newtheorem{theorem}{Theorem}[section]

\newtheorem{remark}{Remark}[section]

\begin{document}
\title[Dispersive estimates with attractive coulomb potential]{\texorpdfstring{{Sharp Dispersive Estimates for the Schr\"odinger Equation with an Attractive Coulomb Potential}}{Sharp Dispersive Estimates for the Schr\"odinger Equation with an Attractive Coulomb Potential}}

\author{Xitao Gao}
\address{Department of Mathematics, Beijing Institute of Technology, Beijing 100081; }
\email{xitao\_gao@bit.edu.cn}

\author{Qiuye Jia}
\address{Department of Mathematics, the Australian National University; }
\email{Qiuye.Jia@anu.edu.au; }

\author{Junyong Zhang}
\address{Department of Mathematics, Beijing Institute of Technology, Beijing 100081; }
\email{zhang\_junyong@bit.edu.cn; }

\begin{abstract}

We prove sharp dispersive $L^1 \to L^\infty$ estimates for the three-dimensional attractive Coulomb operator $H_Z=-\Delta-Z|x|^{-1}$, where $Z>0$. The absolutely continuous part of the Schr\"odinger evolution decays at the free rate for short times, whereas its leading contribution decays like $|t|^{-1}$ for long times, with amplitude proportional to $Z$. This slower decay is driven by the threshold and is sharp when $Z^2|t|\gg1$.

\end{abstract}

\maketitle

\begin{center}
 \begin{minipage}{120mm}
 \end{minipage}
 \end{center}


\tableofcontents

\section{Introduction}

\subsection{Background and motivation}

We consider the three-dimensional Schr\"odinger operator with an attractive Coulomb potential
\begin{align}
\label{eq:HZ-def}
H_Z = -\Delta - \frac{Z}{|x|}, \qquad Z>0, \quad x \in \mathbb{R}^3\setminus\{0\}.
\end{align}
This operator is a standard model in quantum mechanics. In this paper, the normalization in \eqref{eq:HZ-def} is fixed throughout: all spectral and dynamical statements concern the full-Laplacian operator $H_Z=-\Delta-Z|x|^{-1}$.

The free Schr\"odinger equation in $\mathbb{R}^3$ (i.e., without the potential) satisfies the dispersive estimate
\begin{align*}
\| e^{it\Delta} f \|_{L^\infty(\mathbb{R}^3)} \leq C |t|^{-3/2} \| f \|_{L^1(\mathbb{R}^3)}, \qquad t \neq 0,
\end{align*}
which reflects the spreading of the wave packet. For Schr\"odinger operators with potentials, establishing dispersive estimates analogous to the free case is a delicate problem that has attracted considerable attention over the past decades. Early foundational work by Journ\'e, Soffer, and Sogge \cite{JSS91} established decay estimates for a broad class of potentials using resolvent techniques. In one dimension, Goldberg and Schlag \cite{GS04} proved the optimal $|t|^{-1/2}$ decay under suitable decay and regularity assumptions on the potential. In three dimensions, sharp results were obtained by Beceanu and Goldberg \cite{BG12} for generic potentials and by Beceanu \cite{Be16} for exceptional cases involving threshold resonances or eigenvalues. For inverse square potentials $V(x)=a|x|^{-2}$, which decay critically, dispersive estimates have been established in \cite{Fanelli13} for $n=3$ and  \cite{Taira25} for high dimension.

For slowly decaying potentials of the form $V(x) = c |x|^{-\mu}$ with $0 < \mu < 2$, the low-energy behavior differs significantly from the free case, and standard perturbative approaches---such as the Born series---fail due to the slow decay at infinity. In the repulsive Coulomb case $V(x) = + |x|^{-1}$, recent work by Black, Toprak, Vergara, and Zou \cite{BTVZ23} established the dispersive estimate
\begin{align*}
\| e^{-it(-\Delta+|x|^{-1})} f \|_{L^\infty(\mathbb{R}^3)} \leq C |t|^{-3/2} \| f \|_{L^1(\mathbb{R}^3)}, \qquad |t| \geq 1,
\end{align*}
for radially symmetric initial data. Their proof relies on an explicit representation of the distorted Fourier transform in terms of Whittaker functions and a detailed oscillatory integral analysis.

For attractive potentials $V(x)=-c|x|^{-\mu}$ with $0<\mu<2$, the situation is substantially more subtle. Their long-range attractive tail produces infinitely many negative eigenvalues accumulating at zero, which complicates the low-energy spectral structure. In one dimension, Hoshiya and Taira \cite{HT26} recently proved dispersive estimates for Schr\"odinger operators with attractive Coulomb-like potentials satisfying $V(x)\sim-c|x|^{-\mu}$ as $|x|\to\infty$. 
Sussman \cite{Sussman-low} obtained low-energy resolvent asymptotics for smooth attractive Coulomb-like potentials on asymptotically conic manifolds, and investigate the properties of the resolvent output.

In this paper, we establish sharp dispersive estimates in three dimensions for the attractive Coulomb Hamiltonian $H_Z=-\Delta-Z|x|^{-1}$.

\subsection{Main results}
The operator $H_Z$ is understood as the self-adjoint realization associated with its closed quadratic form. Its essential spectrum is $[0,\infty)$ and it has infinitely many negative eigenvalues accumulating at $0$ \cite{RS72}. We denote by $P_{\mathrm{ac}}(H_Z)$ the orthogonal projection onto the absolutely continuous subspace.

\begin{theorem}
\label{thm:main}
Let $Z>0$ and let $H_Z = -\Delta -Z |x|^{-1}$ be the self-adjoint Coulomb Hamiltonian on $L^2(\R^3)$. 
There is an absolute constant $C$ such that, for every $t\in\R\setminus\{0\}$ and for any function $f \in L^1(\mathbb{R}^3)$
\begin{equation}
\label{est:dis}
\| e^{-itH_Z} P_{\mathrm{ac}}(H_Z) f \|_{L^\infty(\mathbb{R}^3)} \leq C\big( |t|^{-3/2} +Z|t|^{-1}\big)\| f \|_{L^1(\mathbb{R}^3)},
\end{equation}
The matching lower bound in the long-time regime is proved in Proposition~\ref{cor:sharp-full}.
\end{theorem}

\begin{remark}
The first term dominates on the free scale $Z^2|t|\lesssim 1$, whereas the second term dominates when $Z^2|t|\gg1$. The estimate is obtained from the exact positive-energy spectral representation, a scaling reduction to a dimensionless parameter, and a uniform estimate for the difference between the Coulomb and free oscillatory profiles. The lower bound is proved by radial data whose low-energy spectral amplitude has a nonzero value at zero. Thus the two-regime law in Theorem~\ref{thm:main} is sharp.
\end{remark}

\begin{remark}
The projection $P_{\mathrm{ac}}(H_Z)$ is necessary because of the bound states; without it, the estimate would fail for large $|t|$ because the eigenfunction component does not decay. The $|t|^{-3/2}$ term is the free-scale contribution, whereas the $Z|t|^{-1}$ term is sharp in the long-time regime.
\end{remark}

\subsection{Organization of the paper}

The paper is organized as follows. Section 2 derives the Coulomb resolvent and the positive-energy spectral measure. Section 3 reduces the propagator estimate to a dimensionless oscillatory integral and proves the uniform direct-difference bound. Section 4 proves the threshold lower bound and the long-time sharpness.

\subsection{Notation}

We write $\langle x \rangle = (1 + |x|^2)^{1/2}$. The notation $f \lesssim g$ means that there exists a constant $C > 0$ such that $f \leq C g$ uniformly in the relevant parameters. For functions of several variables, $f(x) \sim g(x)$ means that $c g(x) \leq f(x) \leq C g(x)$ for some constants $C > c > 0$.

\subsection{Acknowledgement}

Q. Jia is supported by the Australian Research Council through grant FL220100072.
The authors are grateful to helpful conversations with Akitoshi Hoshiya and Kouichi Taira.


\section{The Spectral Measure}\label{sec:spectral-measure}
In this section we derive the positive-energy spectral measure from the exact Coulomb Green function. The geometric variables are
\begin{equation*}
L=|x|+|y|,\qquad R=\sqrt{2\bigl(|x||y|+x\cdot y\bigr)},
\end{equation*}
and, for $\rho>0$,
\begin{equation*}
\ell(\rho)=\log\coth\left(\frac{\rho}{2}\right).
\end{equation*}
They satisfy $0\leq R\leq L$ and $L^2-R^2=|x-y|^2$. The main spectral-measure formula is Proposition~\ref{prop:spectral-measure}.

\subsection{Review of the Coulomb Green's Function} 

In this subsection we review the exact three-dimensional Coulomb Green function in the normalization \eqref{eq:HZ-def}. We write $E_q=-q^2$ and introduce the generalized principal quantum number $\nu=Z/(2q)$, so that $q=Z/(2\nu)$. Thus
\begin{equation*}
(H_Z-E_q)G_Z(E_q;x,y)=\delta(x-y).
\end{equation*}
In spherical coordinates,
\begin{equation*}
\begin{cases}
x_1=r\sin\theta\, \cos\phi,\\
x_2=r\sin\theta\, \sin\phi,\\
x_3=r\cos\theta\,
\end{cases}
\end{equation*}
the Dirac delta distribution has the expansion
\begin{equation*}
\delta(x-y)=\frac{\delta(r_1-r_2)}{r_1^2}\sum_{\ell=0}^\infty \sum_{m=-\ell}^{\ell} Y_{\ell m}(\omega_1)\overline{Y_{\ell m}(\omega_2)},
\end{equation*}
where $x=(r_1,\omega_1)$ and $y=(r_2,\omega_2)$, and $Y_{\ell m}$ are spherical harmonics satisfying
\begin{equation*}
\begin{split}
 \vec{L}^2Y_{\ell m}=-\Delta_{\mathbb{S}^{2}}Y_{\ell m}=\ell(\ell+1)Y_{\ell m}.
\end{split}
\end{equation*}
Expanding the Green function $G_Z(E; x,y)$ in spherical harmonics,
\begin{equation*}
G_Z(E; x, y)=\sum_{\ell=0}^\infty g_{\ell}(E; r_1, r_2)\sum_{m=-\ell}^{\ell}Y_{\ell m}(\omega_1)\overline{Y_{\ell m}(\omega_2)}.
\end{equation*}
then the radial part satisfies
\begin{equation*}
\Big[\frac1{r_1^2}\partial_{r_1}\big(r_1^2\partial_{r_1}\big)-\frac{\ell(\ell+1)}{r_1^2}+\frac{Z}{r_1}-q^2\Big]g_\ell(E_q;r_1,r_2)=-\frac{\delta(r_1-r_2)}{r_1^2}.
\end{equation*}
Equivalently,
\begin{equation*}
E(\nu)=-\frac{Z^2}{4\nu^2}.
\end{equation*}
Let $\omega=\frac{Z}{2\nu}$. The Sturmian expansion of $g_\ell(E(\nu);r_1,r_2)$ is (see, for example, \cite[(2.6)]{SD91})
\begin{equation*}
\begin{split}
 g_\ell(E(\nu); r_1, r_2) &= (2\omega)^{2\ell+1}(r_1r_2)^\ell e^{-\omega(r_1+r_2)} \\
 &\quad\times\sum_{k=0}^{\infty} \frac{k!}{\Gamma(2\ell+2+k)(\ell+1+k-\nu)} L_k^{2\ell+1}(2\omega r_1)L_k^{2\ell+1}(2\omega r_2),
\end{split}
\end{equation*}
where $L^\alpha_{n}(x)$ is a generalized Laguerre polynomial,
\begin{equation*}
L^\alpha_{n}(r)=\frac{\Gamma(\alpha+1+n)}{\Gamma(\alpha+1)n!}\,{}_1F_1(-n;\alpha+1;r)
\end{equation*}
with
\begin{equation*}
(\alpha)_n=\frac{\Gamma(\alpha+n)}{\Gamma(\alpha)},\quad
{}_1F_1(\alpha;\beta;x)=\sum_{n=0}^{\infty}\frac{(\alpha)_n}{(\beta)_n}\frac{x^n}{n!}.
\end{equation*}
This expansion explicitly reveals the poles of the Green function at $\nu=\ell+1+k=n$, corresponding to the bound-state energies
\begin{equation*}
E_n=-\frac{Z^2}{4n^2},\qquad n=1,2,\ldots.
\end{equation*}
Furthermore, the Sturmian expansion yields Hostler's integral representation \cite{Hos64}:
\begin{equation*}
\begin{aligned}
 g_\ell(E(\nu); r_1, r_2)
 &=\frac{1}{\sqrt{r_1 r_2}} \int_0^\infty
 \left(\coth \frac{\rho}{2}\right)^{2\nu}
 \exp\left(-\frac{Z(r_1+r_2)}{2\nu}\cosh \rho\right)\\
 &\qquad\times I_{2\ell+1}\left(\frac{Z\sqrt{r_1 r_2}}{\nu}\sinh \rho\right)\,d\rho.
\end{aligned}
\end{equation*}
To reconstruct the three-dimensional Green function, we insert the spherical-harmonic expansion:
\begin{align*}
 G_Z(E; x, y) =\sum_{\ell=0}^{\infty}g_{\ell}(E(\nu);r_1,r_2)\frac{2\ell+1}{4\pi}P_{\ell}(\hat{x}\cdot\hat{y}),\quad \hat{x}=\frac{x}{|x|}.
\end{align*}
where $P_\ell$ is the Legendre polynomial of degree $\ell$. A key ingredient in summing this series is Neumann's expansion \cite{Hos64}:
\begin{equation*}
\begin{split}
(\tfrac12 kz)^{\mu-\nu} I_{\nu}(kz)=&k^{\mu} \sum_{j=0}^\infty \frac{\Gamma(\mu+j)}{j!\,\Gamma(\nu+1)}(2j+\mu)\\
&\times {_2F_1}(-j,j+\mu,\nu+1;k^2)(-1)^j I_{2j+\mu}(z),
\end{split}
\end{equation*}
with $\mu,\nu,\mu-\nu\neq -1,-2,-3,\ldots$. Applying this with $\mu=1$, $\nu=0$, and the relation
$$P_\ell(z)=(-1)^\ell {}_2F_1(-\ell,\ell+1,1;(1+z)/2),$$ which gives
\begin{align*}
\frac{\rho}{2} I_0\left(\rho \sqrt{\frac{1+\tau}{2}}\right) = \sum_{\ell=0}^\infty (2\ell+1) P_\ell(\tau) I_{2\ell+1}(\rho),
\end{align*}
where $\tau=\hat{x}\cdot\hat{y}$. Substitution into the partial-wave expansion gives the Hostler representation below. Substituting the radial integral representation into the partial-wave expansion and applying the Bessel addition theorem, one obtains the following closed form, which is precisely the Hostler representation in \cite[(2.20)]{SD91}:
\begin{equation*}
\begin{aligned}
&G_Z(E; x, y) \\
&=\frac{Z}{8\pi\nu} \int_0^\infty
 \left(\coth \frac{\rho}{2}\right)^{2\nu}\sinh \rho
 \exp\left(-\frac{Z(|x|+|y|)}{2\nu}\cosh \rho\right)\\
&\qquad\times I_0\left(\frac{Z\sinh\rho}{2\nu}\sqrt{2|x||y|(1+\hat{x}\cdot\hat{y})}\right)\,d\rho\\
&=\frac{Z}{8\pi\nu} \int_0^\infty
 \left(\coth \frac{\rho}{2}\right)^{2\nu}\sinh \rho
 \exp\left(-\frac{ZL}{2\nu}\cosh \rho\right) I_0\left(\frac{Z\sinh\rho}{2\nu}R\right)\,d\rho\\
&=\frac{Z}{8\pi\nu} \int_0^\infty
 \exp(2\nu\ell(\rho))\sinh \rho
 \exp\left(-\frac{ZL}{2\nu}\cosh \rho\right) I_0\left(\frac{Z\sinh\rho}{2\nu}R\right)\,d\rho
\end{aligned}
\end{equation*}
Setting $q=Z/(2\nu)=\sqrt{-E(\nu)}$, we obtain
\begin{equation}
\label{eq:green}
\begin{split}
G_Z(E; x,y) =\frac{q}{4\pi}\int_0^\infty \sinh \rho\,e^{-qL\cosh\rho} I_0\!\left(qR\sinh\rho\right) e^{(Z/q)\ell(\rho)}\,d\rho.
\end{split}
\end{equation}

\begin{proposition}
\label{prop:hostler-resolvent}
 Let $Z>0$. For $q=\alpha-i\beta$ with $\alpha,\beta>0$ and $\operatorname{Re}(Z/(2q))<1$, set
$E_q=-q^2$ (so $\operatorname{Im}E_q>0$), $s(u)=\sqrt{u^2-1}$, and
$\ell_u=\tfrac12\log\tfrac{u+1}{u-1}$. For $x,y\in\R^3$, the Schwartz kernel $G_Z(E_q;x,y)$ of the resolvent $(H_Z-E_q)^{-1}$ is the distribution
\begin{align*}
G_Z(E_q; x,y)
=
\frac{e^{-q|x-y|}}{4\pi|x-y|}
+
\frac{q}{4\pi}\int_1^\infty
e^{-qLu}\,I_0\!\bigl(qR\,s(u)\bigr)
\Bigl[\exp\Bigl(\tfrac{Z}{q}\ell_u\Bigr)-1\Bigr]\,du
\end{align*}
\end{proposition}
\begin{remark}

Put $\nu=Z/(2q)$. Near the endpoint $u=1$, write $u=1+v$. Then
\begin{align*}
\ell_{1+v}=\frac12\log\frac{2}{v}+O(v),\qquad
\exp\!\left(\frac{Z}{q}\ell_{1+v}\right)=\left(\frac{2}{v}\right)^{\nu}(1+O(v)).
\end{align*}
Since $I_0(qR\sqrt{u^2-1})=1+O(v)$, the correction integrand is locally of size $v^{-\operatorname{Re}\nu}$. Therefore the endpoint integral converges precisely when
\begin{align*}
\operatorname{Re}\nu<1,
\qquad\text{equivalently}\qquad
\operatorname{Re}\frac{Z}{2q}<1.
\end{align*}
At infinity, $\ell_u=u^{-1}+O(u^{-3})$, so the free subtraction contributes an additional factor $u^{-1}$; together with $\operatorname{Re}q>0$ this gives convergence and the logarithmic spatial bound used below.

The same condition has a spectral meaning. The poles occur at $\nu=n$ and hence at $E_n=-Z^2/(4n^2)$. For real $q>0$, $\operatorname{Re}\nu<1$ means $q>Z/2$, so $E_q<-Z^2/4$, strictly below the ground-state energy. For complex $q=\alpha-i\beta$ with $\alpha,\beta>0$,
\begin{align*}
\operatorname{Im}E_q=2\alpha\beta>0,
\end{align*}
so self-adjointness places $E_q$ in the resolvent set. In the limiting-absorption limit $q\to-ik$, one has $E_q\to k^2+i0$ and $\nu\to iZ/(2k)$; no bound-state pole is crossed. Thus endpoint integrability and analyticity are controlled by $\operatorname{Re}(Z/(2q))<1$, while $\operatorname{Im}E_q>0$ selects the upper resolvent branch.
\end{remark}

\begin{proof}
This is the Coulomb Green function of Hostler \cite{Hos64} and Swainson--Drake \cite{SD91}. The first term is the free Yukawa resolvent of $-\Delta$ at $E_q$; the factor $\exp((Z/q)\ell_u)-1$ subtracts that free part. From \eqref{eq:green}, we obtain
\begin{equation*}
\begin{aligned}
G_Z(E_q; x,y) &= \frac{q}{4\pi}\int_0^\infty \sinh\rho\,e^{-qL\cosh\rho}I_0\!\left(qR\sinh\rho\right)\,d\rho\\
&+\frac{q}{4\pi}\int_0^\infty \sinh\rho\,e^{-qL\cosh\rho}I_0\!\left(qR\sinh\rho\right)\left[e^{(Z/q)\ell(\rho)}-1\right]d\rho.
\end{aligned}
\end{equation*}
Make the substitution $u=\cosh\rho$, so that $du=\sinh\rho\,d\rho$ and $\sinh\rho = \sqrt{u^2-1}$. Then the first term equals
\begin{align*}
\frac{q}{4\pi}\int_1^\infty e^{-qLu} I_0\!\left(qR\sqrt{u^2-1}\right)du=\frac{e^{-q|x-y|}}{4\pi|x-y|},
\end{align*}
where we use the following standard integral identity, initially for real $q>0$ and then for $\operatorname{Re}q>0$ by analytic continuation,
\begin{align*}
\int_1^\infty e^{-a u} I_0\!\left(b\sqrt{u^2-1}\right) du = \frac{e^{-\sqrt{a^2-b^2}}}{\sqrt{a^2-b^2}}, \quad a > |b|,
\end{align*}
and
\begin{align*}
L^2 - R^2 = (|x|+|y|)^2 - 2|x||y|(1+\hat{x}\cdot\hat{y}) = |x-y|^2.
\end{align*}
Recall the $\theta$-average of the Bessel representation (the Schl\"afli representation of $I_0$),
\begin{align*}I_0(z)=\frac1\pi\int_0^\pi e^{z\cos\theta}\, d\theta\end{align*}
Thus the factor $q/(4\pi)$ in \eqref{eq:green} becomes $q/(4\pi^2)$ after angular averaging. Equivalently, if the angular integral is written over $[0,2\pi]$, its coefficient is $q/(8\pi^2)$. The second term becomes
\begin{equation}\label{eq:DZ}
\begin{aligned}
D_Z(E_q;x,y)&:=G_Z(E_q; x,y)-\frac{e^{-q|x-y|}}{4\pi|x-y|}\\&=\frac{q}{4\pi^2}\int_{0}^{\pi}\int_1^\infty
e^{-qh_\theta(u)}\bigl[B_q(u)-1\bigr]\,du\,d\theta,
\end{aligned}
\end{equation}
where
\begin{align*}\qquad
h_\theta(u)=Lu-Rs(u)\cos\theta,\quad B_q(u)=e^{(Z/q)\ell_u}.\end{align*}
For completeness, we record the correction estimate used in the limiting-absorption argument.
Split the $u$-integral at $u=2$. For $u\ge2$, $\ell_u=u^{-1}+O(u^{-3})$ gives
$|B_q(u)-1|\le C/u$ on compact subsets of the admissible $q$-region. Moreover,
\begin{align*}
h_\theta(u)\ge(L-R)u+\frac{R}{2u}+cRu\theta^2.
\end{align*}
After the $\theta$-integration, the resulting $u$-integral is bounded by
\begin{align*}
\int_2^\infty e^{-c(L-R)u}(1+Ru)^{-1/2}u^{-1}\,du
\le C\bigl(1+|\log L|\bigr),
\end{align*}
with the logarithm occurring only at the joint Coulomb singularity $L=R=0$. On the compact interval
$1\le u\le2$, the endpoint estimates for the modified Bessel functions, valid uniformly for $q$
in compact subsets of $\{\operatorname{Re}q>0\}$,
\begin{align*}
|I_0((\alpha-i\beta)w)|\le Ce^{\alpha w}(1+w)^{-1/2},\qquad
|I_1((\alpha-i\beta)w)|\le Ce^{\alpha w}\min\{w,(1+w)^{-1/2}\},
\end{align*}
bound the $u$-integrand by an integrable function on $[1,2]$. The same estimates, with the Abel
regulator retained, pass to the boundary $q\to-ik$ and give
\begin{align*}
|D_Z(E_q;x,y)|\le C\bigl(1+|\log L|\bigr)
\end{align*}
uniformly on compact spatial subsets in the corresponding Abel sense; the only spatial singularity is
the logarithmic Coulomb singularity at $L=0$. The verification
$(H_Z+q^2)G_Z=\delta(x-y)$ is the standard one for this representation (the
off-diagonal differential identity together with a boundary-flux argument at $x=y$ and $x=0$, using the
gradient bound $|\nabla_x D_Z(E_q; x,y)|\le C(1/L+1/|x-y|)$ from the same $I_0,I_1$ estimates); see
\cite{Hos64,SD91}.
\end{proof}

\subsection{The spectral measure}
\label{subsec:spectral-measure}
In this subsection we use the Green function above to derive the spectral-measure representation. We first construct the boundary values of the resolvent.
\begin{proposition}[Boundary values of the resolvent]\label{prop:boundary-values}
For $k>0$, the upper and lower boundary values
\begin{align*}
G_Z^\pm(k;x,y)
=
\lim_{\epsilon\downarrow0}\bigl(H_Z-(k^2\pm i\epsilon)\bigr)^{-1}(x,y),
\end{align*}
are the distributions
\begin{align*}
G_Z^+(k;x,y)
=-\frac{ik}{4\pi}
\left[
\int_0^\infty
\sinh(\rho)\,
J_0\!\bigl(kR\sinh(\rho)\bigr)
\exp\!\left(i kL\cosh(\rho)+i\frac{Z}{k}\ell(\rho)\right)
\,d\rho
\right],
\end{align*}
and $G_Z^-(k;x,y)=\overline{G_Z^+(k;y,x)}$.
Here $J_0$ is the Bessel function of order zero.
\end{proposition}

\begin{proof} We first consider $G_Z^+(k;x,y)$.
Take $q=\alpha-i\beta$ with $\beta\to k$ and $\alpha\downarrow0$, so that $q\to-ik$ and
$$E_q=-q^2\to k^2$$ from the upper half-plane. By Proposition~\ref{prop:hostler-resolvent}, the
free term tends to the outgoing free Green function $$\frac{e^{ik|x-y|}}{4\pi|x-y|},$$ and the correction,
with $q$ replaced by $-ik$, is\footnote{The notation here is a bit different from \eqref{eq:DZ}, here we replace $E_q$ by $q=-ik$ to emphasize the role of $k$.}
\begin{align*}
D_Z(-ik;x,y)
=
-\frac{ik}{4\pi}\int_1^\infty
e^{ikLu}\,J_0\!\bigl(kR\,s(u)\bigr)
\Bigl[\exp\Bigl(\tfrac{iZ}{k}\ell_u\Bigr)-1\Bigr]\,du ,
\end{align*}
the convergence being distributional in the common free-subtracted Abel prescription (the raw
integral is not asserted to converge pointwise; this is exactly why the free subtraction is
retained).

The free term is locally integrable in $(x,y)\in\R^6$, and the Abel boundary estimate for the
correction is locally integrable as well. Consequently, after pairing with a compactly supported
test function in $(x,y)$, the Abel limit may be taken under the pairing. This is the precise sense in
which the displayed boundary value is understood.

The free term is the value of the same $u$-integral with the factor $1$ in place of
$\exp((iZ/k)\ell_u)$ (equivalently, the free outgoing Green function
$e^{ik|x-y|}/(4\pi|x-y|)$ has the Hostler representation
$$-\frac{ik}{4\pi}\int_1^\infty e^{ikLu}J_0(kR s(u))\,du$$ (i.e., the $Z=0$ case of
Proposition~\ref{prop:hostler-resolvent})), 
so recombining the free term with the correction assigns
the single full Abel representative
\begin{align*}
G_Z^+(k;x,y)
=-\frac{ik}{4\pi}\int_1^\infty
e^{ikLu}\,J_0\!\bigl(kR\,s(u)\bigr)
\exp\Bigl(\tfrac{iZ}{k}\ell_u\Bigr)\,du .
\end{align*}
Substituting $u=\cosh(\rho)$ (whence $du=\sinh(\rho)\,d\rho$, $s(u)=\sinh(\rho)$, and
$\ell_u=\ell(\rho)=\log\coth(\rho/2)$) yields the displayed $\rho$-integral. Because the coefficient
of $\ell(\rho)$ in the exponent is $+iZ/k$, the attractive Coulomb logarithmic phase is
$+(Z/k)\ell(\rho)$. Finally the lower boundary value is the adjoint of the upper one; since
$L$, $R$, $\sinh(\rho)$, and $\ell(\rho)$ are symmetric in $x$ and $y$, it equals
$\overline{G_Z^+(k;y,x)}$.
\end{proof}

\begin{proposition}
\label{prop:spectral-decomp}
For $Z>0$, the positive spectral part of the Coulomb Hamiltonian $H_Z=-\Delta-Z|x|^{-1}$ on $L^2(\R^3)$ is purely absolutely continuous. Moreover, on its absolutely continuous subspace, $H_Z$ is unitarily equivalent to the operator of multiplication by $k^2$ ($k>0$) via the Hankel--Whittaker transforms of Derezinski and Richard \cite[Theorem 3.16 and Eq.~(3.43)]{DR18}.
\end{proposition}

We next derive the spectral-measure representation.
\begin{proposition}[Spectral measure]
\label{prop:spectral-measure}
For $k>0$, the positive continuous spectral measure of $H_Z$ is the distribution
\begin{equation}\label{eq:spect}
dE(k;x,y)
=-\frac{k^2}{2\pi^2}
\left[
\int_0^\infty
\sinh(\rho)\,
J_0\!\bigl(kR\sinh(\rho)\bigr)
\cos\!\left(kL\cosh(\rho)+\frac{Z}{k}\ell(\rho)\right)
\,d\rho
\right]dk .
\end{equation}
In particular, when $Z=0$, formula \eqref{eq:spect} reduces to the spectral measure of $-\Delta$.
\end{proposition}
\begin{remark}
The fact that the positive spectral part of $H_Z$ is purely absolutely continuous, and hence that the measure above is the absolutely continuous spectral measure, follows from the Derezi\'nski--Richard diagonalization in Proposition~\ref{prop:spectral-decomp}.
\end{remark}
\begin{proof}
By Proposition~\ref{prop:boundary-values},
\begin{align*}
G_Z^-(k;x,y)
=
\frac{ik}{4\pi}
\left[
\int_0^\infty
\sinh(\rho)\,
J_0\!\bigl(kR\sinh(\rho)\bigr)
\exp\!\left(-i kL\cosh(\rho)-i\frac{Z}{k}\ell(\rho)\right)
\,d\rho
\right].
\end{align*}
Writing
\begin{align*}
\Phi(k;\rho)=kL\cosh(\rho)+\frac{Z}{k}\ell(\rho)
\end{align*}
and subtracting the two boundary values yields
\begin{align*}
G_Z^+(k;x,y)-G_Z^-(k;x,y)
=-\frac{ik}{2\pi}
\left[
\int_0^\infty
\sinh(\rho)\,
J_0\!\bigl(kR\sinh(\rho)\bigr)\cos(\Phi)\,d\rho
\right].
\end{align*}
For the convention $G_Z(E)=(H_Z-E)^{-1}$, Stone's formula gives, at energy $E=k^2$,
\begin{align*}
dE(k^2)
= \frac{1}{2\pi i}\bigl(G_Z^+(E)-G_Z^-(E)\bigr)\,dk^2.
\end{align*}
Since $E=k^2$ and $dE=2k\,dk$, the scalar coefficient is
\begin{align*}
\frac{1}{2\pi i}\left(-\frac{ik}{2\pi}\right)(2k)
=-\frac{k^2}{2\pi^2},
\end{align*}
which proves the formula \eqref{eq:spect}. 

When $Z=0$, \eqref{eq:spect} becomes
\begin{equation*}
\begin{split}
dE(k;x,y)
&=-\frac{k^2}{2\pi^2}
\left[
\int_0^\infty
\sinh(\rho)\,
J_0\!\bigl(kR\sinh(\rho)\bigr)
\cos\!\left(kL\cosh(\rho)\right)
\,d\rho
\right]dk\\
&=\frac{k}{2\pi^2}\frac{\sin(k|x-y|)}{|x-y|}\,dk
\end{split}
\end{equation*}
which agrees with the spectral measure of $-\Delta$.
\begin{equation*}
\begin{split}
dE(k;x,y)&=\frac{k^2}{(2\pi)^3}\int_{\mathbb{S}^2} e^{ik\omega\cdot(x-y)}\,d\omega \, dk
 =\frac{k}{2\pi^2}\frac{\sin(k|x-y|)}{|x-y|}\,dk.
\end{split}
\end{equation*}
Here we used
\begin{align*}
\int_{\mathbb{S}^2}e^{ik\omega\cdot z}\,d\omega
=4\pi\frac{\sin(k|z|)}{k|z|},
\qquad z=x-y,
\end{align*}
so the factor $4\pi$ from the spherical integral is essential for the coefficient $k/(2\pi^2)$.

\end{proof}

\section{Dispersive estimates for the Schr\"odinger propagator}\label{sec:est-Schrodinger-propagator}

We prove the pointwise estimate for the Schr\"odinger propagator. Since $H_Z$ is self-adjoint, the kernel at negative time is the complex conjugate transpose of the kernel at positive time; it therefore suffices to work with $t>0$. The argument below is uniform in the coupling and in the long-time parameter.

\subsection{Reduction to estimating the difference $\mathfrak{D}$}
Let $H_Z$ be as in \eqref{eq:HZ-def} and let $H_0=-\Delta$. We define the kernel difference
\begin{equation*}
\mathfrak{D}(t;x,y):=(e^{-itH_Z}P_{\mathrm{ac}}(H_Z))(x,y)-e^{-itH_0}(x,y).
\end{equation*}
By Proposition \ref{prop:spectral-measure}, we obtain
\begin{equation}\label{eq:kernel-difference-rho}
\begin{aligned}
&\mathfrak{D}(t;x,y)\\
  &=-\frac{1}{2\pi^2}\int_0^\infty k^2e^{-itk^2}\int_0^\infty
  \sinh(\rho)\,J_0\!\bigl(kR\sinh(\rho)\bigr)\\
  &\qquad\times\Big[\cos\!\Bigl(kL\cosh(\rho)+\frac{Z}{k}\ell(\rho)\Bigr)
  -\cos\!\bigl(kL\cosh(\rho)\bigr)\Big] \,d\rho\,dk.
\end{aligned}
\end{equation}
For $t>0$, set
\begin{align}
\label{eq:scaled-variables}
\overline{x}=\frac{x}{\sqrt{2t}}, \quad &\overline{y}=\frac{y}{\sqrt{2t}}, \quad a=Z\sqrt{2t}, \\ \mathrm{L}=|\overline{x}|+|\overline{y}|=L/\sqrt{2t}, 
\quad &\mathrm{R}= \sqrt{2(|\overline{x}||\overline{y}|+\overline{x}\cdot\overline{y})}=R/\sqrt{2t}. \nonumber
\end{align}
Then
\begin{align*}
\mathrm{L}^2-\mathrm{R}^2=|\overline{x}-\overline{y}|^2,
\end{align*}
and hence $0\leq \mathrm{R}\leq \mathrm{L}$. Make the change of variable $p=\sqrt{2t}\,k$. Then
\begin{align}
k^2\,dk=(2t)^{-3/2}p^2\,dp,\qquad e^{-itk^2}=e^{-ip^2/2}.
\label{eq:momentum-scaling}
\end{align}
By \eqref{eq:scaled-variables}, we have
\begin{align}
|x|+|y|=\sqrt{2t}\,\mathrm{L},\qquad \sqrt{2(|x||y|+x\cdot y)}=\sqrt{2t}\,\mathrm{R},\qquad \frac{Z}{k}\ell(\rho)=\frac{a}{p}\ell(\rho).
\label{eq:scaled-phases}
\end{align}
Substituting \eqref{eq:momentum-scaling} and \eqref{eq:scaled-phases} into \eqref{eq:kernel-difference-rho} gives
\begin{align}
\mathfrak{D}(t;x,y)
&=-\frac{(2t)^{-3/2}}{2\pi^2}\int_0^\infty p^2e^{-ip^2/2}\int_0^\infty \sinh\rho\,J_0(p\mathrm{R}\sinh\rho) \nonumber\\
&\qquad\times\left[\cos(p\mathrm{L}\cosh\rho+\frac{a}{p}\log\coth\frac{\rho}{2})-\cos(p\mathrm{L}\cosh\rho)\right]\,d\rho\,dp.
\label{eq:scaled-kernel-difference}
\end{align}

Setting $s=\ell(\rho)=\log\coth\frac{\rho}{2}$, we have
\begin{align}
\cosh\rho=\coth s,\qquad \sinh\rho=\operatorname{csch}s,\qquad \sinh\rho\,d\rho=-\operatorname{csch}^2s\,ds.
\label{eq:hostler-substitution}
\end{align}
As $\rho$ increases from $0$ to $\infty$, $s$ decreases from $\infty$ to $0$, so reversing the endpoints cancels the minus sign in \eqref{eq:hostler-substitution}.
We denote the right-hand side of \eqref{eq:scaled-kernel-difference}, up to the rescaling, by\footnote{This notation is different from $\mathfrak{D}(t;x,y)$.}
\begin{align}
\mathfrak{D}(a,\mathrm{L},\mathrm{R})
&=\int_0^\infty\int_0^\infty p^2\operatorname{csch}^2s\,e^{-ip^2/2}J_0(p\mathrm{R}\operatorname{csch}s) \nonumber\\
&\qquad\times\left[\cos(p\mathrm{L}\coth s+\frac{as}{p})-\cos(p\mathrm{L}\coth s)\right]\,ds\,dp.
\label{eq:direct-difference}
\end{align}
All oscillatory integrals in this section are understood in the common ordered Abel sense. More precisely, for $\varepsilon,\sigma>0$ and $0<\delta<N$, we first insert the factor
\begin{align*}
e^{-\varepsilon(s+s^{-1})}e^{-\sigma p^2}
\end{align*}
and replace the $p$-integration by $\int_\delta^N$; denote the resulting absolutely convergent integral by $\mathfrak{D}_{\varepsilon,\delta,N,\sigma}(a,\mathrm{L},\mathrm{R})$. The unregularized expression in \eqref{eq:direct-difference} denotes the ordered limit
\begin{align*}
\lim_{\sigma\downarrow0}\lim_{\substack{\delta\downarrow0\\N\uparrow\infty}}
\lim_{\varepsilon\downarrow0}
\mathfrak{D}_{\varepsilon,\delta,N,\sigma}(a,\mathrm{L},\mathrm{R}).
\end{align*}
All differentiations, Fubini interchanges, and changes of variables below are first performed at finite regulators; the estimates are then passed to this ordered limit.
The free Schr\"odinger kernel satisfies
$$\big| e^{-itH_0}(x,y)\big|\leq C|t|^{-3/2}.$$ 
Therefore,
\begin{align}\label{eq:positive-time-reduction}
\big| (e^{-itH_Z}P_{\mathrm{ac}}(H_Z))(x,y)\big|&\leq \big| e^{-itH_0}(x,y)\big|+C|t|^{-\frac32}|\mathfrak{D}(a,\mathrm{L},\mathrm{R})|\\
&\leq C|t|^{-\frac32}\big(1+|\mathfrak{D}(a,\mathrm{L},\mathrm{R})|\big). \nonumber
\end{align}
Therefore, recalling the rescaled variables in \eqref{eq:scaled-variables}, in particular
$a=Z\sqrt{2|t|}$, it suffices to prove
\begin{align}
|\mathfrak{D}(a,\mathrm{L},\mathrm{R})|\leq Ca,\qquad a>0,\qquad 0\leq \mathrm{R}\leq \mathrm{L},
\label{eq:mkD-est}
\end{align}
with a constant uniform in $a,\mathrm{L},\mathrm{R}$. The remainder of this section proves this estimate. We use $\mathrm{L}$ and $\mathrm{R}$ throughout the scaled argument.

\subsection[The case L=0]{{The case $\mathrm{L}=0$}}
We first prove \eqref{eq:mkD-est} when $\mathrm{L}=0$. Then $\mathrm{R}=0$, and the quantity to be estimated is
\begin{align}\label{Da00}
\mathfrak{D}(a,0,0)
&=\int_0^\infty\int_0^\infty p^2\operatorname{csch}^2s\,e^{-ip^2/2}\left[\cos(\frac{as}{p})-1\right]\,ds\,dp.
\end{align}
where we use $J_0(0)=1$.

\begin{proposition}Let $\mathfrak{D}(a,0,0)$ be defined by \eqref{Da00}. Then there exists a constant $C$ independent of $a$ such that
\begin{align}\label{eq:mkD-est'}
|\mathfrak{D}(a,0,0)|\leq Ca.
\end{align}

\end{proposition}

\begin{proof}
To estimate \eqref{eq:mkD-est'}, make the change of variables $p=su$ and $y=s^2$. Then $dp\,ds=\tfrac12\,du\,dy$,
$p^2/2=y u^2/2$, and, with $A(y)=y\operatorname{csch}^2(\sqrt y)$,
\begin{align*}
p^2\operatorname{csch}^2(s)\,dp\,ds=\tfrac12 u^2 A(y)\,du\,dy .
\end{align*}
We consider the integration in $y$ first. Define the endpoint Fresnel transform
\begin{align*}
W_0(u)=\frac12\int_0^\infty A(y)\,e^{-iyu^2/2}\,dy .
\end{align*}
Since $A$ is smooth with $A(0)=1$ and $A,A'\in L^1([0,\infty))$ (both decay like a polynomial in
$\sqrt y$ times $e^{-2\sqrt y}$), one has $\sup_{u>0}u^2|W_0(u)|\le C_0$. Indeed, for $0<u\le1$,
$$|u^2W_0(u)|\le\tfrac12 u^2\|A\|_1\le C_0,$$ while for $u\ge1$ integration by parts using
$$e^{-iyu^2/2}=(2i/u^2)\partial_y e^{-iyu^2/2}$$ gives $|W_0(u)|\le(A(0)+\|A'\|_1)/u^2\le C_0/u^2$.

We next consider the integration in $u$. The region $p\le as$ (that is, $u\le a$) contributes
\begin{align*}
G_{<}(a)=\int_0^a u^2W_0(u)\bigl[\cos(a/u)-1\bigr]\,du ,
\end{align*}
with $|\cos(a/u)-1|\le2$, so $|G_{<}(a)|\le2C_0a$. The region $p\ge as$ ($u\ge a$) contributes
\begin{align*}
G_{>}(a)=\int_a^\infty u^2W_0(u)\bigl[\cos(a/u)-1\bigr]\,du ,
\end{align*}
and here $a/u\le1$ gives $|\cos(a/u)-1|\le\tfrac12(a/u)^2$, so
$$|G_{>}(a)|\le\tfrac{C_0}{2}\int_a^\infty a^2 u^{-2}\,du=\tfrac{C_0}{2}a.$$ Hence
$|\mathfrak{D}(a,0,0)|\le C a$ which proves \eqref{eq:mkD-est'}.
\end{proof}

\begin{remark}\label{rem:L0}
For the remainder of this section we assume $\mathrm{L}>0$,
so that $q=\mathrm{R}/\mathrm{L}$ is defined and $0\le q\le1$; the case $\mathrm{R}=0$ (i.e.\ $q=0$) is included in the $\mathrm{L}>0$
analysis below. 
\end{remark}

\subsection{The estimates \eqref{eq:mkD-est} when $\mathrm{L}>0$}
In the case $\mathrm{L}>0$, it is enough to prove a uniform bound for the $a$-derivative of the common Abel-regularized integral and then pass to the limit. We compute this derivative first.
\begin{lemma}
Let $0\leq \mathrm{R}\leq \mathrm{L}$, and $a\geq 0$. Let $\mathfrak{D}$ be as in \eqref{eq:direct-difference}.
Then
\begin{align}
\label{eq:direct-difference-derivative}
\frac{d}{da}\mathfrak{D}(a,\mathrm{L},\mathrm{R})
=-\frac{1}{4\pi}\int_0^{2\pi}\widehat P_\theta(a)\,d\theta,
\end{align}
where 
\begin{align}
\begin{split}
\widehat{P}_\theta(a)
= & \int_0^\infty
\nu \int_0^\infty y\operatorname{csch}^2(\sqrt y)e^{-iy\nu^2/2}  \sin\left(
\nu\sqrt y\frac{\mathrm{L}\cosh\sqrt y+\mathrm{R}\cos\theta}{\sinh\sqrt y}
+\frac a\nu
\right)dyd\nu.
\end{split}
\label{eq:fixed-angle-paired-sine}
\end{align}
\end{lemma}

\begin{proof}
Differentiating \eqref{eq:direct-difference} with respect to $a$ gives\footnote{Rigorously, one should regularize the integrand by introducing $e^{-\epsilon s^2 - \epsilon p^2}$ first and then send $\epsilon\to 0$, which is standard and we omit this step here and below.}
\begin{align}
\begin{split}
\frac{d}{da}\mathfrak{D}(a,\mathrm{L},\mathrm{R})
={}&-\int_0^\infty\int_0^\infty
ps\operatorname{csch}^2(s)e^{-ip^2/2}J_0(p\mathrm{R}\operatorname{csch}s) \\
&\quad\times
\sin\left(p\mathrm{L}\coth s+\frac{as}{p}\right)\,ds\,dp.
\end{split}
\label{eq:direct-difference-first-derivative}
\end{align}
Recall that
\begin{align*}
J_0(b)
=\frac{1}{2\pi}\int_0^{2\pi}e^{ib\cos\theta}\,d\theta
=\frac{1}{2\pi}\int_0^{2\pi}\cos(b\cos\theta)\,d\theta. 
\end{align*}
We have
\begin{align}
\begin{split}
&\frac{1}{2\pi}\int_0^{2\pi}\sin(C+b\cos\theta)\,d\theta
\\={}&\sin C\frac{1}{2\pi}\int_0^{2\pi}\cos(b\cos\theta)\,d\theta +\cos C\frac{1}{2\pi}\int_0^{2\pi}\sin(b\cos\theta)\,d\theta 
\\
=&J_0(b) \sin C,
\end{split}
\label{eq:bessel-sine-identity}
\end{align}
where we used $\int_0^{2\pi}\sin(b\cos\theta)\,d\theta=0$.
Using \eqref{eq:bessel-sine-identity} with $$b=p\mathrm{R}\operatorname{csch}s,\quad  C=p\mathrm{L}\coth s+\frac{as}{p},$$ then \eqref{eq:direct-difference-first-derivative} becomes
\begin{align*}
\begin{split}
\frac{d}{da}\mathfrak{D}(a,\mathrm{L},\mathrm{R})
={}&-\frac{1}{2\pi}
\int_0^{2\pi}\int_0^\infty\int_0^\infty
ps\operatorname{csch}^2(s)e^{-ip^2/2} \\
&\quad\times
\sin\left(
p\mathrm{L}\coth s+p\mathrm{R}\operatorname{csch}s\cos\theta+\frac{as}{p}
\right)\,ds\,dp\,d\theta.
\end{split}
\end{align*}
Now we make the change of variables
\begin{align*}
y=s^2, \quad
\nu=\frac{p}{s},
\end{align*}
to obtain
\begin{align*}
\begin{split}
\frac{d}{da}\mathfrak{D}(a,\mathrm{L},\mathrm{R})
={}&-\frac{1}{4\pi}
\int_0^{2\pi}\int_0^\infty
\nu\int_0^\infty y\operatorname{csch}^2(\sqrt y)e^{-iy\nu^2/2} 
\\
&\quad\times
\sin\left(
\nu\sqrt y\frac{\mathrm{L}\cosh\sqrt y+\mathrm{R}\cos\theta}{\sinh\sqrt y} + \frac a\nu
\right)\,dy\,d\nu\,d\theta,
\end{split}
\end{align*}
which is \eqref{eq:direct-difference-derivative}.
\end{proof}

We now estimate $\frac{d}{da}\mathfrak{D}(a,\mathrm{L},\mathrm{R})$ using its oscillatory structure. Writing $\sin(\bullet)=(2i)^{-1}(e^{i\bullet}-e^{-i\bullet})$ gives the two phases.
\begin{align}
\Phi_\pm(y,\nu)=-\frac{y\nu^2}{2}\pm\left(\mathrm{L}\nu K_z(y)+\frac{a}{\nu}\right).
\label{eq:phase-branches}
\end{align}
where
\begin{align}
z=\frac{\mathrm{R}}{\mathrm{L}}\cos\theta\in[-1,1], \quad K_z(y)=\frac{\sqrt y(\cosh\sqrt y+z)}{\sinh\sqrt y}\, (y>0), \quad K_z(0)=1+z.
\label{eq:Kz-definition}
\end{align}

The following elementary fact controls the critical-point geometry of the scaled oscillatory integrals.
\begin{proposition}
\label{prop:phase-prop}
Let $z\in[-1,1]$ and $y\geq0$. Let $\Phi_\pm(y,\nu)$ be given by \eqref{eq:phase-branches}, and let $K_z(y)$ be as in \eqref{eq:Kz-definition}.
Then $K_z\in C^\infty([0,\infty))$,
\begin{align*}
K_z'(y)>0,\qquad K_z''(y)<0\qquad (y>0),
\end{align*}
\begin{align*}
K_z'(0)=\frac{2-z}{6},
\qquad
\lim_{y\to\infty}K_z'(y)=0.
\end{align*}
Consequently, the phase $\Phi_-$ has no $y$-critical point, and the phase $\Phi_+$ has only one critical point $y\in(0,\infty)$ when $0<\nu<2\mathrm{L}K'_z(0)$, has the critical point $y=0$ when $\nu=2\mathrm{L}K'_z(0)$, and has no critical point when $\nu>2\mathrm{L}K'_z(0)$.

\end{proposition}

\begin{proof}
Write $s=\sqrt y$ and $u=s/2$. At the two endpoint values of $z$,
\begin{align*}
K_1(y)=s\coth(u),\qquad K_{-1}(y)=s\tanh(u)=2u\tanh(u),
\end{align*}
and
\begin{align*}
K_z=\frac{1+z}{2}K_1+\frac{1-z}{2}K_{-1}.
\end{align*}
For $K_{-1}$, differentiation with respect to $y$ gives
\begin{align*}
K_{-1}'(y)=\frac14\left(\frac{\tanh u}{u}+\operatorname{sech}^2u\right)>0.
\end{align*}
The derivative of the expression in parentheses is
\begin{align*}
\frac{u\operatorname{sech}^2u-\tanh u}{u^2}-2\operatorname{sech}^2u\tanh u<0,
\end{align*}
because $\tanh u-u\operatorname{sech}^2u>0$ for $u>0$. Hence $K_{-1}''(y)<0$.

For $K_1$,
\begin{align*}
K_1'(y)=\frac14\left(\frac{\coth u}{u}-\operatorname{csch}^2u\right)>0,
\end{align*}
since $\sinh u\cosh u>u$. If
\begin{align*}
h(u)=\frac{\coth u}{u}-\operatorname{csch}^2u,
\end{align*}
then
\begin{align*}
h'(u)=\frac{-u-\sinh u\cosh u+2u^2\coth u}{u^2\sinh^2u}.
\end{align*}
The inequalities
\begin{align*}
\coth u<\frac1u+\frac u3,
\qquad
\sinh u\cosh u>u+\frac23u^3
\end{align*}
imply $h'(u)<0$, and therefore $K_1''(y)<0$. The first inequality follows by differentiating
$(3+u^2)\sinh u-3u\cosh u$, and the second follows by differentiating
$\sinh u\cosh u-u-\frac23u^3$.

Convex interpolation now yields $K_z'(y)>0$ and $K_z''(y)<0$ for every $z\in[-1,1]$ and $y>0$. The expansions
\begin{align*}
\coth(s/2)=\frac2s+\frac s6+O(s^3),
\qquad
\tanh(s/2)=\frac s2-\frac{s^3}{24}+O(s^5)
\end{align*}
give $K_z(0)=1+z$ and $K_z'(0)=(2-z)/6$. They also show that $K_z$ is smooth as a function of $y$ at $0$. Finally,
\begin{align*}
K_z(y)=\sqrt y\bigl(1+O(e^{-\sqrt y})\bigr),
\qquad
K_z'(y)=\frac1{2\sqrt y}+O(e^{-\sqrt y}),
\end{align*}
as $y\to\infty$, so $K_z'(y)\to0$. In addition, direct differentiation gives
\begin{align*}
K_z(y)-2yK_z'(y)=\frac{y(1+z\cosh\sqrt y)}{\sinh^2\sqrt y},
\end{align*}
and the monotonicity property in \eqref{eq:Kz-property} follows by differentiating this expression on the region where it is nonnegative.
For reference, the derivative formula is
\begin{align*}
\begin{split}
K'_z(y)
&=\frac{\sinh\sqrt y(\cosh\sqrt y+z)-\sqrt y(1+z\cosh\sqrt y)}{2\sqrt y\sinh^2\sqrt y}.
\end{split}
\end{align*}
The identities used below are
\begin{align}
\label{eq:K'z-positive}
K'_z(y)>0.
\end{align}
Moreover, a direct computation shows
\begin{align}
\begin{split}
K'_z(y)>0,
\qquad K''_z(y)<0,
\qquad &K'_z(0)=\frac{2-z}{6},
\qquad \lim_{y\to\infty}K'_z(y)=0,\\
K_z(y)-2yK'_z(y)
&=
\frac{y(1+z\cosh\sqrt y)}{\sinh^2\sqrt y},\\
\frac{d}{dy}
\big(
4K'_z(y)^2\big(K_z(y)-2yK'_z(y)\big)
\big)
&<0
\quad\text{whenever}\quad
K_z(y)-2yK'_z(y)\geq0.
\end{split}
\label{eq:Kz-property}
\end{align}

Since $a/\nu$ is independent of $y$, the $y$-derivatives of these phases are
\begin{align}
\partial_y\Phi_+(y,\nu)=\nu\left(\mathrm{L}K'_z(y)-\frac{\nu}{2}\right), \qquad
\partial_y\Phi_-(y,\nu)=-\nu\left(\mathrm{L}K'_z(y)+\frac{\nu}{2}\right).
\label{eq:phase-derivatives}
\end{align}

By \eqref{eq:K'z-positive}, $\partial_y\Phi_-(y,\nu)<0$ for every $y\geq0$ and $\nu>0$, so $\Phi_-$ has no $y$-critical point. For $\Phi_+$, the critical-point equation is
\begin{align}
\partial_y\Phi_+(y,\nu)=0 \quad\Longleftrightarrow\quad \nu=2\mathrm{L}K'_z(y).
\label{eq:critical-equation}
\end{align}

Since $K'_z$ is strictly decreasing from $K'_z(0)$ to $0$, equation \eqref{eq:critical-equation} has exactly one solution $y\in(0,\infty)$ when $0<\nu<2\mathrm{L}K'_z(0)$, has the boundary solution $y=0$ when $\nu=2\mathrm{L}K'_z(0)$, and has no solution when $\nu>2\mathrm{L}K'_z(0)$. 

\end{proof}

The critical frequency from Proposition~\ref{prop:phase-prop} is
$\nu=2\mathrm{L}K'_z(0)$. We use the enlarged splitting point
\begin{align*}
\nu = 4\mathrm{L}K'_z(0)=\frac{2}{3}(2\mathrm{L}-\mathrm{R}\cos\theta),
\end{align*}
which has distance to the critical point of the phase comparable to its own size. Then, from \eqref{eq:fixed-angle-paired-sine}, we have
\begin{align}
\begin{split}
\widehat P_\theta(a)={}&\int_0^{4\mathrm{L}K'_z(0)}\nu\int_0^\infty y\operatorname{csch}^2(\sqrt y)e^{-iy\nu^2/2}\sin(\mathrm{L}\nu K_z(y)+a/\nu)\,dy\,d\nu\\
&+\lim_{V\to\infty}\int_{4\mathrm{L}K'_z(0)}^V\nu\int_0^\infty y\operatorname{csch}^2(\sqrt y)e^{-iy\nu^2/2}\sin(\mathrm{L}\nu K_z(y)+a/\nu)\,dy\,d\nu.
\end{split}
\label{eq:Phat-decomposition}
\end{align}

\subsection[Low-frequency contribution]{Low-frequency contribution}
We now consider the first integral in \eqref{eq:Phat-decomposition}; we will prove the required angularly integrated bound.

\begin{lemma} \label{lem:ps}
Let $A(y)=y\operatorname{csch}^2(\sqrt y)$.
For every $\ell>0$, $\rho\geq0$, and $q\in[0,1]$ with $z=q\cos\theta$, define
\begin{align}
\widehat P^\ell(\rho;z) = \int_0^{4\ell K'_z(0)}
\nu\int_0^\infty
A(y)e^{-iy\nu^2/2}
\sin\big(\ell\nu K_z(y)+\rho/\nu\big)
\,dy\,d\nu.
\label{eq:ps-def}
\end{align}
Then there is an absolute constant $C$ such that
\begin{align}
\Big| \int_0^{2\pi}
\widehat P^\ell(\rho;q\cos\theta)\,d\theta
\Big| \leq C.
\label{eq:ps-est}
\end{align}
\end{lemma}

\begin{proof}
The main tool for proving this lemma is the one-dimensional van der Corput lemma and stationary phase in their standard forms; see \cite[Chapter VIII, \S1.2, Proposition~2 and the following corollary, pp.~332--334; \S2.3, Proposition~6, pp.~344--345]{Stein93}.

A direct computation shows
\begin{align}
\begin{split}
& A,A',A'' \in L^1(0,\infty), \quad A(s^2) = 4s^2e^{-2s}\big(1+O(e^{-2s})\big),\\
& K'_z(s^2) = \frac{1}{2s}+O(e^{-s}), \quad 
K''_z(s^2) = -\frac{1}{4s^3}+O(e^{-s}/s).
\end{split}
\label{eq:ps-asym}
\end{align}
If $0<\ell<1$, then \eqref{eq:Kz-property} gives $4\ell K'_z(0)\leq2\ell$, and hence
\begin{align}
\begin{split}
|\widehat P^\ell(\rho;z)|
&\leq
\int_0^{4\ell K'_z(0)}\nu\,d\nu
\int_0^\infty A(y)\,dy
\leq C\ell^2
\leq C.
\end{split}
\label{eq:ps-smell}
\end{align}
Integrating \eqref{eq:ps-smell} in $\theta$ proves \eqref{eq:ps-est} for $0<\ell<1$.

From now on, we assume $\ell\geq1$. Set
\begin{align*}
\Lambda=\ell^2,
\qquad
\widetilde\rho=\frac{\rho}{\ell^3},
\qquad
\nu=\ell w.
\end{align*}
Then \eqref{eq:ps-def} becomes
\begin{align}
\begin{split}
\widehat P^\ell(\rho;z)
=
\Lambda
\int_0^{4K'_z(0)}
w\int_0^\infty
A(y)e^{-i\Lambda yw^2/2}
\sin\Big(
\Lambda\big(wK_z(y)+\frac{\widetilde\rho}{w}\big)
\Big)
\,dy\,dw.
\end{split}
\label{eq:ps-scaled}
\end{align}
The large parameter in \eqref{eq:ps-scaled} is
\begin{align}
\Lambda=\ell^2\geq1.
\label{eq:ps-lam}
\end{align}
Define
\begin{align}
\Phi_\pm(y,w;z)
=
-\frac{yw^2}{2}
\pm\big(wK_z(y)+\frac{\widetilde\rho}{w}\big).
\label{eq:ps-ph}
\end{align}
By \eqref{eq:ps-scaled} and \eqref{eq:ps-ph},
\begin{align}
\begin{split}
\widehat P^\ell(\rho;z)
=
\frac{\Lambda}{2i}
\int_0^{4K'_z(0)}
w\int_0^\infty
A(y)
\big(
e^{i\Lambda\Phi_+(y,w;z)}
-
e^{i\Lambda\Phi_-(y,w;z)}
\big)
\,dy\,dw.
\end{split}
\label{eq:ps-dec}
\end{align}

For the minus phase,
\begin{align}
\partial_y\Phi_-(y,w;z)
=
-w\big(\frac{w}{2}+K'_z(y)\big)<0.
\label{eq:ps-md}
\end{align}
One integration by parts in $y$, using \eqref{eq:ps-lam}, \eqref{eq:ps-md} and $A(0)=1$, gives
\begin{align}
\begin{split}
&
|
\Lambda
\int_0^{4K'_z(0)}
w\int_0^\infty
A(y)e^{i\Lambda\Phi_-(y,w;z)}
\,dy\,dw
|
\\
&\leq
\int_0^{4K'_z(0)}
\frac{dw}{w/2+K'_z(0)}
+
\int_0^\infty
|A'(y)|
\int_0^{4K'_z(0)}
\frac{dw}{w/2+K'_z(y)}
\,dy
\\ & \quad+
\int_0^\infty
A(y)(-K''_z(y))
\int_0^{4K'_z(0)}
\frac{dw}{\big(w/2+K'_z(y)\big)^2}dy.
\end{split}
\label{eq:ps-mibp}
\end{align}
The three $w$-integrals in \eqref{eq:ps-mibp} satisfy
\begin{align}
\begin{split}
& \int_0^{4K'_z(0)}
\frac{dw}{w/2+K'_z(0)}
 \lesssim 1, \quad \int_0^{4K'_z(0)}
\frac{dw}{w/2+K'_z(y)}
\leq C\log(2+\sqrt y),
\\ & (-K''_z(y))
\int_0^{4K'_z(0)}
\frac{dw}{\big(w/2+K'_z(y)\big)^2}
\leq\frac{C}{1+y}.
\end{split}
\label{eq:ps-mw}
\end{align}
Equations \eqref{eq:ps-asym}, \eqref{eq:ps-mibp}, and \eqref{eq:ps-mw} yield
\begin{align}
\sup_{\Lambda\geq1,\ \widetilde\rho\geq0,\ z\in[-1,1]}
\Big|
\Lambda
\int_0^{4K'_z(0)}
w\int_0^\infty
A(y)e^{i\Lambda\Phi_-(y,w;z)}
\,dy\,dw\Big|
\leq C.
\label{eq:ps-mbd}
\end{align}

We next consider the contribution of the phase $\Phi_+$.

\subsubsection{The phase $\Phi_+$ with $w\in[0,w_0]$. } Fix $w_0\in(0,1/12)$ sufficiently small. For $z\in[-1,-1/2]$, choose $S_0\geq2$ so that
\begin{align}
1+z\cosh s\leq-c\cosh s
\qquad
(s\geq S_0),
\label{eq:ps-neg0}
\end{align}
which will be used later in the region $\sqrt{y}=s\geq S_0$.
By \eqref{eq:Kz-property}, we know $K'_z(y)>0$ and $\lim_{y\to\infty}K'_z(y)=0$, which gives a constant $c_0$ such that $K'_z(y) \geq c_0$ when $y \in [0,S_0^2]$.
After decreasing $w_0$,  this gives
\begin{align}
\partial_y\Phi_+(y,w;z)
=
w\big(K'_z(y)-\frac{w}{2}\big)
\geq cw
\label{eq:ps-negc}
\end{align}
for $0\leq y\leq S_0^2$, $0<w\leq w_0$, and $z\in[-1,-1/2]$. One integration by parts in $y$, using \eqref{eq:ps-lam} and \eqref{eq:ps-negc}, gives
\begin{align}
\sup_{\Lambda\geq1,\ \widetilde\rho\geq0,\ z\in[-1,-1/2]}
|
\Lambda
\int_0^{w_0}
w\int_0^{S_0^2}
A(y)e^{i\Lambda\Phi_+(y,w;z)}
\,dy\,dw
|
\leq C.
\label{eq:ps-negcb}
\end{align}
On $y=s^2\geq S_0^2$, set $u=sw$. The transformed integral is
\begin{align}
\begin{split}
&
\Lambda
\int_0^{w_0}
w\int_{S_0^2}^\infty
A(y)e^{i\Lambda\Phi_+(y,w;z)}
\,dy\,dw
\\
&=
\Lambda
\int_0^\infty
\int_{\max\{S_0,u/w_0\}}^\infty
2us\operatorname{csch}^2s
\exp\big(
i\Lambda\big(
-\frac{u^2}{2}
+u\frac{\cosh s+z}{\sinh s}
+\frac{\widetilde\rho s}{u}
\big)
\big)
\,ds\,du.
\end{split}
\label{eq:ps-negt}
\end{align}
By \eqref{eq:ps-neg0}, the $s$-derivative of the phase in \eqref{eq:ps-negt} satisfies
\begin{align}
-u\frac{1+z\cosh s}{\sinh^2s}
+\frac{\widetilde\rho}{u}
\geq cue^{-s}+\frac{\widetilde\rho}{u}.
\label{eq:ps-negs}
\end{align}
Moreover,
\begin{align}
\begin{split}
&
\frac{
2us\operatorname{csch}^2s
}{
-u(1+z\cosh s)/\sinh^2s+\widetilde\rho/u
}
\\
&\quad+
\Big|
\partial_s
\big(
\frac{
2us\operatorname{csch}^2s
}{
-u(1+z\cosh s)/\sinh^2s+\widetilde\rho/u
}
\big)
\Big|
\leq C(1+s)e^{-s}.
\end{split}
\label{eq:ps-negq}
\end{align}
One integration by parts in $s$, using \eqref{eq:ps-lam}, \eqref{eq:ps-negs}, and \eqref{eq:ps-negq}, gives
\begin{align}
\begin{split}
&
|
\Lambda
\int_0^{w_0}
w\int_{S_0^2}^\infty
A(y)e^{i\Lambda\Phi_+(y,w;z)}
\,dy\,dw
|
\\
&\leq
C
\int_0^\infty
\big(1+\max\{S_0,u/w_0\}\big)
e^{-\max\{S_0,u/w_0\}}
\,du
\leq C.
\end{split}
\label{eq:ps-negb}
\end{align}
Equations \eqref{eq:ps-negcb} and \eqref{eq:ps-negb} imply
\begin{align}
\sup_{\Lambda\geq1,\ \widetilde\rho\geq0,\ z\in[-1,-1/2]}
\Big|
\Lambda
\int_0^{w_0}
w\int_0^\infty
A(y)e^{i\Lambda\Phi_+(y,w;z)}
\,dy\,dw
\Big|
\leq C.
\label{eq:ps-neg}
\end{align}

For $z\in[-1/2,1]$ and $0<w\leq w_0$, the equation
\begin{align}
w=2K'_z(y_c)
\label{eq:ps-yc}
\end{align}
has exactly one solution $y_c=y_c(w,z)>0$. After decreasing $w_0$, one has $y_c\geq S_0^2$. In the coordinate $y=y_cr^2$, the critical point is $r=1$, and
\begin{align*}
\partial_r^2
\Big(
-\frac{y_cr^2w^2}{2}+wK_z(y_cr^2)
\Big)
\big|_{r=1}
=
4y_c^2wK''_z(y_c).
\end{align*}
Equations \eqref{eq:ps-asym} and \eqref{eq:ps-yc} give
\begin{align*}
-c_2\leq4y_c^2wK''_z(y_c)\leq-c_1<0.
\end{align*}
Applying the stationary phase in the $r$-variable with the large parameter \eqref{eq:ps-lam} gives
\begin{align}
\begin{split}
&
\Lambda w
\int_0^\infty
A(y)e^{i\Lambda(-yw^2/2+wK_z(y))}
\,dy
\\
&=
c_0\sqrt\Lambda\,
A(y_c)
\sqrt{\frac{w}{-K''_z(y_c)}}
e^{i\Lambda(-y_cw^2/2+wK_z(y_c))}
+R(w,z,\Lambda),
\end{split}
\label{eq:ps-ysp}
\end{align}
where
\begin{align}
\sup_{z\in[-1/2,1],\ \Lambda\geq1}
\int_0^{w_0}|R(w,z,\Lambda)|\,dw
\leq C.
\label{eq:ps-yr}
\end{align}
The bound \eqref{eq:ps-yr} follows by applying the first-derivative form of the van der Corput lemma in the $r$-variable outside a fixed neighborhood of $r=1$ and using \eqref{eq:ps-asym}.

Parameterize \eqref{eq:ps-yc} by
\begin{align*}
y_c=s^2,
\qquad
w=2K'_z(s^2).
\end{align*}
The phase of the leading term in \eqref{eq:ps-ysp}, including the reciprocal term in \eqref{eq:ps-ph}, is
\begin{align}
S_{z,\widetilde\rho}(s)
=
-2s^2K'_z(s^2)^2
+2K'_z(s^2)K_z(s^2)
+\frac{\widetilde\rho}{2K'_z(s^2)}.
\label{eq:ps-S}
\end{align}
The change of variables $w=2K'_z(s^2)$ transforms the leading contribution of \eqref{eq:ps-ysp} into
\begin{align}
\begin{split}
c_0\sqrt\Lambda
\int_{S_0}^\infty
4sA(s^2)
\sqrt{2K'_z(s^2)(-K''_z(s^2))}
e^{i\Lambda S_{z,\widetilde\rho}(s)}
\,ds.
\end{split}
\label{eq:ps-Sint}
\end{align}
Differentiating \eqref{eq:ps-S} gives
\begin{align}
\begin{split}
S'_{z,\widetilde\rho}(s)
=
\frac{sK''_z(s^2)}{K'_z(s^2)^2}
\big(
4K'_z(s^2)^2
\big(K_z(s^2)-2s^2K'_z(s^2)\big)
-\widetilde\rho
\big).
\end{split}
\label{eq:ps-Sd}
\end{align}
The exact identities
\begin{align}
\begin{split}
K_z(s^2)
&=
s\frac{1+2ze^{-s}+e^{-2s}}{1-e^{-2s}},\\
K_z(s^2)-2s^2K'_z(s^2)
&=
\frac{2s^2e^{-s}\big(z(1+e^{-2s})+2e^{-s}\big)}
{(1-e^{-2s})^2}
\end{split}
\label{eq:ps-ex}
\end{align}
and direct differentiation give, after increasing $S_0$,
\begin{align}
\begin{split}
&
\big|
4sA(s^2)
\sqrt{2K'_z(s^2)(-K''_z(s^2))}
\big|
\\
&\quad+
\Big|
\frac{d}{ds}
\big(
4sA(s^2)
\sqrt{2K'_z(s^2)(-K''_z(s^2))}
\big)
\Big|
\leq C(1+s)e^{-2s},
\end{split}
\label{eq:ps-B}\\
\begin{split}
\frac12
\leq
-\frac{sK''_z(s^2)}{K'_z(s^2)^2}
\leq2,
\qquad
\Big|
\frac{d}{ds}
\frac{sK''_z(s^2)}{K'_z(s^2)^2}
\Big|
\leq C(1+s)^2e^{-s}.
\end{split}
\nonumber
\end{align}
Equations \eqref{eq:Kz-property}, \eqref{eq:ps-Sd}, and \eqref{eq:ps-ex} imply that each interval $[j,j+1]$, $j\geq S_0$, is the union of at most an absolute number of intervals $J$ such that either
\begin{align}
|S''_{z,\widetilde\rho}(s)|\geq ce^{-2j}
\qquad
(s\in J),
\label{eq:ps-cell2}
\end{align}
with constant sign, or
\begin{align}
|S'_{z,\widetilde\rho}(s)|\geq ce^{-2j}
\qquad
(s\in J),
\label{eq:ps-cell1}
\end{align}
with $S'_{z,\widetilde\rho}$ monotone. The function
\begin{align*}
4K'_z(s^2)^2
\big(K_z(s^2)-2s^2K'_z(s^2)\big)
\end{align*}
is strictly decreasing while it is nonnegative by \eqref{eq:Kz-property}. When it is negative, its difference from $\widetilde\rho\geq0$ cannot vanish. The factor $z(1+e^{-2s})+2e^{-s}$ in \eqref{eq:ps-ex} has at most one zero. Direct differentiation of \eqref{eq:ps-ex} shows that failure of \eqref{eq:ps-cell1} implies \eqref{eq:ps-cell2}.

On an interval satisfying \eqref{eq:ps-cell2}, the second-derivative form of the van der Corput lemma, together with \eqref{eq:ps-lam} and \eqref{eq:ps-B}, gives
\begin{align}
\sqrt\Lambda
| \int_J
4sA(s^2)
\sqrt{2K'_z(s^2)(-K''_z(s^2))}
e^{i\Lambda S_{z,\widetilde\rho}(s)}
ds | \leq C(1+j)e^{-j}.
\label{eq:ps-cell2b}
\end{align}
On an interval satisfying \eqref{eq:ps-cell1}, the first-derivative form of the van der Corput lemma (so $-1$ power of the lower bound in \eqref{eq:ps-cell1} enters), together with \eqref{eq:ps-lam} and \eqref{eq:ps-B}, gives
\begin{align}
\begin{split}
&
\sqrt\Lambda
| \int_J
4sA(s^2) \sqrt{2K'_z(s^2)(-K''_z(s^2))}
e^{i\Lambda S_{z,\widetilde\rho}(s)} \,ds |
\leq \frac{C(1+j)}{\sqrt\Lambda}.
\end{split}
\label{eq:ps-cell1b}
\end{align}
Summing \eqref{eq:ps-cell2b} over $j\geq S_0$ and summing \eqref{eq:ps-cell1b} over
\begin{align*}
j\leq\frac13\log(2+\Lambda)
\end{align*}
give
\begin{align}
\sum_J
\sqrt\Lambda
|
\int_J
4sA(s^2)
\sqrt{2K'_z(s^2)(-K''_z(s^2))}
e^{i\Lambda S_{z,\widetilde\rho}(s)}
\,ds
|
\leq C
\label{eq:ps-cells}
\end{align}
for the intervals contained in $S_0\leq s\leq\frac13\log(2+\Lambda)$. On $s\geq\frac13\log(2+\Lambda)$, \eqref{eq:ps-B} gives
\begin{align}
\sqrt\Lambda
\int_{\frac13\log(2+\Lambda)}^\infty
(1+s)e^{-2s}\,ds
\leq C.
\label{eq:ps-Stail}
\end{align}
Equations \eqref{eq:ps-yr}, \eqref{eq:ps-Sint}, \eqref{eq:ps-cells}, and \eqref{eq:ps-Stail} imply
\begin{align}
\sup_{\Lambda\geq1,\ \widetilde\rho\geq0,\ z\in[-1/2,1]}
\Big|
\Lambda
\int_0^{w_0}
w\int_0^\infty
A(y)e^{i\Lambda\Phi_+(y,w;z)}
\,dy\,dw
\Big|
\leq C.
\label{eq:ps-pos}
\end{align}
Equations \eqref{eq:ps-neg} and \eqref{eq:ps-pos} control the plus phase for $0<w\leq w_0$ and every $z\in[-1,1]$.

\subsubsection{The phase $\Phi_+$ with $w_0\leq w\leq4K'_z(0)$. } Since $K'_z(y)>0$ and $\lim_{y\to\infty}K'_z(y)=0$, we can choose $Y>0$ so that
\begin{align}
K'_z(y)\leq\frac{w_0}{4}
\qquad
(y\geq Y,\ z\in[-1,1]).
\label{eq:ps-Y}
\end{align}
For $y\geq Y$ and $w_0\leq w\leq4K'_z(0)$, equations \eqref{eq:ps-ph} and \eqref{eq:ps-Y} give
\begin{align}
|\partial_y\Phi_+(y,w;z)|
=
w\Big(\frac{w}{2}-K'_z(y)\Big)
\geq\frac{w^2}{4}.
\label{eq:ps-yd}
\end{align}
One integration by parts in $y$, using \eqref{eq:ps-lam}, \eqref{eq:ps-asym}, and \eqref{eq:ps-yd}, gives
\begin{align}
\sup_{\Lambda\geq1,\ \widetilde\rho\geq0,\ z\in[-1,1]}
|
\Lambda
\int_{w_0}^{4K'_z(0)}
w\int_Y^\infty
A(y)e^{i\Lambda\Phi_+(y,w;z)}
\,dy\,dw
|
\leq C.
\label{eq:ps-ybd}
\end{align}

Consider $z\in[-3/4,1]$ and $0\leq y\leq2Y$. The critical equations for \eqref{eq:ps-ph} are
\begin{align}
\begin{split}
w&=2K'_z(y),\\
\widetilde\rho
&=
4K'_z(y)^2\big(K_z(y)-2yK'_z(y)\big).
\end{split}
\label{eq:ps-c2}
\end{align}
Equation \eqref{eq:Kz-property} shows that \eqref{eq:ps-c2} has at most one solution. At an interior solution of \eqref{eq:ps-c2},
\begin{align}
\begin{split}
\det D^2_{y,w}\Phi_+(y,w;z)
=
\frac{1}{4K'_z(y)}
\frac{d}{dy}
\big(
4K'_z(y)^2\big(K_z(y)-2yK'_z(y)\big)
\big)
<0.
\end{split}
\label{eq:ps-h2}
\end{align}
At the boundary $y=0$, the critical values are
\begin{align*}
\begin{split}
w&=\frac{2-z}{3},
\qquad
\widetilde\rho=\frac{(1+z)(2-z)^2}{9},
\end{split}
\end{align*}
and the Hessian determinant is
\begin{align}
\begin{split}
\det D^2_{y,w}\Phi_+(0,w;z)
&=\frac{z^2+2z-4}{20}.
\end{split}
\label{eq:ps-h20}
\end{align}
The determinant in \eqref{eq:ps-h20} is uniformly separated from zero for $z\in[-3/4,1]$. Equations \eqref{eq:ps-c2}, \eqref{eq:ps-Y}, \eqref{eq:ps-h2}, and \eqref{eq:ps-h20} give a uniform inverse-Hessian bound for every critical point in
\begin{align*}
0\leq y\leq2Y,
\qquad
w_0\leq w\leq4K'_z(0),
\qquad
z\in[-3/4,1].
\end{align*}
For bounded $\widetilde\rho$, stationary phase in $(y,w)$ with the large parameter \eqref{eq:ps-lam} gives a factor $\Lambda^{-1}$. This cancels the exterior factor $\Lambda$ in \eqref{eq:ps-dec}, and therefore
\begin{align}
\sup_{\substack{\Lambda\geq1,\ 0\leq\widetilde\rho\leq C_0,\\ z\in[-3/4,1]}}
\Big|
\Lambda
\int_{w_0}^{4K'_z(0)}
w\int_0^{2Y}
A(y)e^{i\Lambda\Phi_+(y,w;z)}
\,dy\,dw
\Big|
\leq C.
\label{eq:ps-mid0}
\end{align}
For $\widetilde\rho\geq C_0$ with $C_0$ sufficiently large,
\begin{align}
\partial_w\Phi_+(y,w;z)
=
-yw+K_z(y)-\frac{\widetilde\rho}{w^2}
\label{eq:ps-rhod}
\end{align}
satisfies
\begin{align}
|\partial_w\Phi_+(y,w;z)|\geq c\widetilde\rho
\label{eq:ps-rhol}
\end{align}
on the compact domain of \eqref{eq:ps-mid0}. One integration by parts in $w$, using \eqref{eq:ps-lam}, \eqref{eq:ps-rhod}, and \eqref{eq:ps-rhol}, gives
\begin{align}
\sup_{\substack{\Lambda\geq1,\ \widetilde\rho\geq C_0,\\ z\in[-3/4,1]}}
\Big|
\Lambda
\int_{w_0}^{4K'_z(0)}
w\int_0^{2Y}
A(y)e^{i\Lambda\Phi_+(y,w;z)}
\,dy\,dw
\Big|
\leq C.
\label{eq:ps-mid1}
\end{align}
Equations \eqref{eq:ps-ybd}, \eqref{eq:ps-mid0}, and \eqref{eq:ps-mid1} imply
\begin{align}
\sup_{\Lambda\geq1,\ \widetilde\rho\geq0,\ z\in[-3/4,1]}
\Big|
\Lambda
\int_{w_0}^{4K'_z(0)}
w\int_0^\infty
A(y)e^{i\Lambda\Phi_+(y,w;z)}
\,dy\,dw
\Big|
\leq C.
\label{eq:ps-mid}
\end{align}

Since the angular set $q\cos\theta<-3/4$ is empty for $q\leq3/4$, we may assume $q>3/4$. Write
\begin{align*}
\theta=\pi\pm\varphi,
\qquad
z=-q\cos\varphi,
\qquad
0\leq\varphi\leq\arccos\big(3/(4q)\big).
\end{align*}
Define
\begin{align}
\Phi_{q,\widetilde\rho}(y,w,\varphi)
=
-\frac{yw^2}{2}
+wK_{-q\cos\varphi}(y)
+\frac{\widetilde\rho}{w}.
\label{eq:ps-ap}
\end{align}
The derivatives of \eqref{eq:ps-ap} are
\begin{align}
\begin{split}
\partial_y\Phi_{q,\widetilde\rho}
&= w \big(K'_{-q\cos\varphi}(y)-\frac{w}{2}\big),
\quad  \partial_w\Phi_{q,\widetilde\rho}
= -yw+K_{-q\cos\varphi}(y)-\frac{\widetilde\rho}{w^2}, 
\\ & \partial_\varphi\Phi_{q,\widetilde\rho} =
wq\sin\varphi\,
\frac{\sqrt y}{\sinh\sqrt y}.
\end{split}
\label{eq:ps-apd}
\end{align}
Hence every critical point of \eqref{eq:ps-ap} satisfies
\begin{align}
\begin{split}
\varphi&=0, \quad w = 2K'_{-q}(y), \quad 
\widetilde\rho =
4K'_{-q}(y)^2\big(K_{-q}(y)-2yK'_{-q}(y)\big).
\end{split}
\label{eq:ps-apc}
\end{align}
Since
\begin{align*}
K_{-q}(y)-2yK'_{-q}(y)
=
\frac{y(1-q\cosh\sqrt y)}{\sinh^2\sqrt y},
\end{align*}
every solution of \eqref{eq:ps-apc} satisfies
\begin{align}
0\leq y
\leq
\big(\operatorname{arcosh}(1/q)\big)^2
\leq
\big(\operatorname{arcosh}(4/3)\big)^2.
\label{eq:ps-apy}
\end{align}
At an interior critical point, all mixed second derivatives containing exactly one $\varphi$-derivative vanish, and
\begin{align}
\begin{split}
\det D^2_{y,w,\varphi}\Phi_{q,\widetilde\rho}
=
\frac{wq\sqrt y}{\sinh\sqrt y}
\frac{1}{4K'_{-q}(y)}
\frac{d}{dy}
\big(
4K'_{-q}(y)^2
\big(K_{-q}(y)-2yK'_{-q}(y)\big)
\big)
\neq0.
\end{split}
\label{eq:ps-h3}
\end{align}
At $y=0$,
\begin{align}
\begin{split}
w&=\frac{2+q}{3},\\
\det D^2_{y,w,\varphi}\Phi_{q,\widetilde\rho}(0,w,0)
&=
\frac{wq(q^2-2q-4)}{20}.
\end{split}
\label{eq:ps-h30}
\end{align}
Equations \eqref{eq:Kz-property}, \eqref{eq:ps-apy}, \eqref{eq:ps-h3}, and \eqref{eq:ps-h30} give a uniform inverse-Hessian bound for $q\in[3/4,1]$.

For bounded $\widetilde\rho$, stationary phase in $(y,w,\varphi)$ with the large parameter \eqref{eq:ps-lam} gives a factor $\Lambda^{-3/2}$. Multiplication by the exterior factor $\Lambda$ in \eqref{eq:ps-dec} gives
\begin{align*}
\Lambda
\Big|
\iiint
e^{i\Lambda\Phi_{q,\widetilde\rho}(y,w,\varphi)}
A(y)w
\,dy\,dw\,d\varphi
\Big|
\leq C\Lambda^{-1/2}
\leq C
\end{align*}
on a fixed neighborhood of the critical set \eqref{eq:ps-apc}. On the complement of that fixed neighborhood in the compact domain determined by \eqref{eq:ps-apy}, the gradient in \eqref{eq:ps-apd} is uniformly separated from zero, and the identity
\begin{align*}
e^{i\Lambda\Phi_{q,\widetilde\rho}}
=
\frac{\nabla\Phi_{q,\widetilde\rho}}
{i\Lambda|\nabla\Phi_{q,\widetilde\rho}|^2}
\cdot
\nabla e^{i\Lambda\Phi_{q,\widetilde\rho}}
\end{align*}
gives an $O(1)$ contribution after multiplication by the exterior factor $\Lambda$. For sufficiently large $\widetilde\rho$, the second equation in \eqref{eq:ps-apd} satisfies
\begin{align}
|\partial_w\Phi_{q,\widetilde\rho}|
\geq c\widetilde\rho
\label{eq:ps-apr}
\end{align}
on the compact $(y,w,\varphi)$-domain, and one integration by parts in $w$, using \eqref{eq:ps-lam} and \eqref{eq:ps-apr}, gives an $O(1)$ contribution. The domain $y\geq Y$ is controlled by \eqref{eq:ps-ybd}. Therefore
\begin{align}
\begin{split}
&
\Big|
\int_{\{\theta:q\cos\theta<-3/4\}}
\Lambda
\int_{w_0}^{4K'_{q\cos\theta}(0)}
w\int_0^\infty
A(y)e^{i\Lambda\Phi_+(y,w;q\cos\theta)}
\,dy\,dw\,d\theta
\Big|
\leq C.
\end{split}
\label{eq:ps-apbd}
\end{align}

The minus branch is bounded by \eqref{eq:ps-mbd}. The plus branch on $0<w\leq w_0$ is bounded by \eqref{eq:ps-neg} and \eqref{eq:ps-pos}. The plus branch on $w_0\leq w\leq4K'_z(0)$ and $z\in[-3/4,1]$ is bounded by \eqref{eq:ps-mid}. The angularly integrated plus branch on $w_0\leq w\leq4K'_z(0)$ and $z<-3/4$ is bounded by \eqref{eq:ps-apbd}. Substitution of \eqref{eq:ps-mbd}, \eqref{eq:ps-neg}, \eqref{eq:ps-pos}, \eqref{eq:ps-mid}, and \eqref{eq:ps-apbd} into \eqref{eq:ps-dec} gives
\begin{align}
\Big|
\int_0^{2\pi}
\widehat P^\ell(\rho;q\cos\theta)\,d\theta
\Big|
\leq C
\qquad
(\ell\geq1).
\label{eq:ps-lg}
\end{align}
Equations \eqref{eq:ps-smell} and \eqref{eq:ps-lg} prove \eqref{eq:ps-est}.
\end{proof}

\subsection[High-frequency contribution]{High-frequency contribution}
In the region $\nu \geq 4\mathrm{L}K'_z(0)$, we have $|\mathrm{L}K'_z(y)-\frac{\nu}{2}| \geq \frac{\nu}{4}$.
Combining this with \eqref{eq:phase-derivatives}, we have
\begin{align*}
|\partial_y\Phi_+(y,\nu)|\geq\frac{\nu^2}{4}, \qquad
|\partial_y\Phi_-(y,\nu)|\geq\frac{\nu^2}{2}.
\end{align*}

Next we consider the second integral in \eqref{eq:Phat-decomposition}. We will prove that there is a constant $C>0$, independent of all parameters, such that
\begin{align}
\begin{split}
\Bigg|\int_0^{2\pi}\lim_{V\to\infty}\int_{4\mathrm{L}K'_z(0)}^V\nu\int_0^\infty y\operatorname{csch}^2(\sqrt y)e^{-iy\nu^2/2}\sin(\mathrm{L}\nu K_z(y)+a/\nu)\,dy\,d\nu\,d\theta\Bigg| \leq C.
\end{split}
\label{eq:Phat-second-bound}
\end{align}
To this end, we first record the following one-dimensional oscillatory estimate.

\begin{lemma}\label{lem:uniform-oscillatory-tails}
Let $A,B,C\geq0$ and $U>0$ satisfy $U\geq2C$. Define
\begin{align}
\begin{split}
E(A,B;U) = \int_U^\infty\sin(A\xi+B/\xi)\frac{d\xi}{\xi},
\\ G(A,B,C;U) = \int_U^\infty\frac{\sin(A\xi+B/\xi)}{\xi-C}\,d\xi.
\end{split}
\label{eq:uniform-oscillatory-tails}
\end{align}
Then the improper integrals in \eqref{eq:uniform-oscillatory-tails} exist and satisfy
\begin{align*}
|E(A,B;U)|+|G(A,B,C;U)|\leq C_0,
\end{align*}
where $C_0$ is a constant independent of $A,B,C,U$.
\end{lemma}
\begin{proof}
We first estimate $E(A,B;U)$. If $A=B=0$, then $E(A,B;U)=0$. If $A=0$ and $B>0$, the substitution $t=B/\xi$ gives
\begin{align*}
E(0,B;U)=\int_0^{B/U}\frac{\sin t}{t}\,dt.
\end{align*}
If $A>0$ and $B=0$, the substitution $t=A\xi$ gives
\begin{align*}
E(A,0;U)=\int_{AU}^\infty\frac{\sin t}{t}\,dt.
\end{align*}
For $R\geq r\geq1$, integration by parts gives
\begin{align*}
\int_r^R\frac{\sin t}{t}\,dt=\frac{\cos r}{r}-\frac{\cos R}{R}-\int_r^R\frac{\cos t}{t^2}\,dt,
\end{align*}
and therefore
\begin{align}
\big|\int_r^R\frac{\sin t}{t}\,dt\big|\leq\frac{3}{r}.
\label{eq:sine-integral-bound}
\end{align}
This proves the required bound when $A=0$ or $B=0$.
Assume now that $A,B>0$, and set
\begin{align*}
\xi_*=\sqrt{B/A},\quad c=\sqrt{AB},\quad u=U/\xi_*.
\end{align*}
Changing the variables by $\xi=\xi_* x$, we have
\begin{align}
E(A,B;U)=\int_u^\infty\sin\big(c(x+x^{-1})\big)\frac{dx}{x}.
\label{eq:uniform-tail-scaled}
\end{align}
Suppose first that $u\geq1$. If $0<c \leq 1$, then
\begin{align*}
|\sin\big(c(x+x^{-1})\big)-\sin(cx)|\leq\frac{c}{x},
\end{align*}
and consequently
\begin{align*}
\Big|E(A,B;U)-\int_u^\infty\frac{\sin(cx)}{x}\,dx\Big|\leq c\int_u^\infty\frac{dx}{x^2}=\frac{c}{u}\leq1.
\end{align*}
The substitution $t=cx$ gives
\begin{align*}
\int_u^\infty\frac{\sin(cx)}{x}\,dx=\int_{cu}^\infty\frac{\sin t}{t}\,dt,
\end{align*}
which is uniformly bounded by \eqref{eq:sine-integral-bound} and $|\sin t|/t\leq1$.

Now we consider the case $c\geq1$, and set
\begin{align*}
\psi(x)=c(x+x^{-1}).
\end{align*}
We have  
\begin{align*}
\psi'(x)=c(1-x^{-2}),\quad \psi''(x)=2cx^{-3}.
\end{align*}
For $1\leq x\leq2$, one has $\psi''(x)\geq c/4$ and $\psi'(1)=0$. Therefore the subset of $u\leq x\leq2$ on which $\psi'(x)\leq\sqrt c$ has length at most $4c^{-1/2}$ and contributes at most $4c^{-1/2}$ to \eqref{eq:uniform-tail-scaled}. On its complement, $\psi'(x)\geq\sqrt c$, and $x\psi'(x)=c(x-x^{-1})$
is strictly increasing. 
Integrating by parts, for $\alpha<\beta$, we have
\begin{align}
\begin{split}
\int_\alpha^\beta\frac{\sin\psi(x)}{x}\,dx&=\frac{\cos\psi(\alpha)}{\alpha\psi'(\alpha)}-\frac{\cos\psi(\beta)}{\beta\psi'(\beta)} +\int_\alpha^\beta\cos\psi(x)\frac{d}{dx}\big(\frac{1}{x\psi'(x)}\big)\,dx.
\end{split}
\label{eq:uniform-tail-ibp}
\end{align}
Since $1/(x\psi'(x))$ is positive and decreasing, \eqref{eq:uniform-tail-ibp} is bounded in absolute value by $3c^{-1/2}$. For $x\geq2$, one has $\psi'(x)\geq3c/4$. Applying \eqref{eq:uniform-tail-ibp} from $\max(u,2)$ to $R$ and letting $R\to\infty$ gives a bound $C_0c^{-1}$. Consequently, \eqref{eq:uniform-tail-scaled} is uniformly bounded when $u\geq1$.
If $0<u<1$, the substitution $x\mapsto x^{-1}$ in the integral from $u$ to $1$ gives
\begin{align*}
\int_u^1\sin\big(c(x+x^{-1})\big)\frac{dx}{x}=\int_1^{1/u}\sin\big(c(x+x^{-1})\big)\frac{dx}{x}.
\end{align*}
Consequently, we have
\begin{align*}
E(A,B;U)=2\int_1^\infty\sin\big(c(x+x^{-1})\big)\frac{dx}{x}-\int_{1/u}^\infty\sin\big(c(x+x^{-1})\big)\frac{dx}{x}.
\end{align*}
The estimate for \eqref{eq:uniform-tail-scaled} with the lower bound $\geq1$ applies to both integrals. Consequently, we obtain
\begin{align*}
|E(A,B;U)|\leq C_0.
\end{align*}
It remains to estimate $G(A,B,C;U)$ in \eqref{eq:uniform-oscillatory-tails}. Since $U\geq2C$ and $\frac{1}{\xi-C}=\frac{1}{\xi}+\frac{C}{\xi(\xi-C)}\quad (\xi\geq U)$, we have
\begin{align*}
G(A,B,C;U)=E(A,B;U)+\int_U^\infty\frac{C\sin(A\xi+B/\xi)}{\xi(\xi-C)}\,d\xi.
\end{align*}
For $\xi\ge U\ge2C$,
\begin{align*}
0\le \frac{C}{\xi(\xi-C)}\le\frac{2C}{\xi^2},
\end{align*}
and therefore
\begin{align*}
\left|\int_U^\infty\frac{C\sin(A\xi+B/\xi)}{\xi(\xi-C)}\,d\xi\right|
\le\frac{2C}{U}\le1.
\end{align*}
The estimate for $E(A,B;U)$ now gives $|G(A,B,C;U)|\lesssim 1$.
\end{proof}
Now we prove the estimate controlling the integrand in \eqref{eq:Phat-second-bound} for fixed $\theta$, which implies \eqref{eq:Phat-second-bound}.

\begin{lemma}\label{lem:ordered-high-frequency-tail}
Let $-1\leq z\leq1$, $\ell>0$, $a\geq0$, and $U>0$ satisfy
\begin{align*}
U\geq4\ell K'_z(0).
\end{align*}
Define the amplitude by
\begin{align*}
A(y)=y\operatorname{csch}^2(\sqrt y),\quad A(0)=1.
\end{align*}
Then the ordered tail
\begin{align}
\begin{split}
P_U(a,\ell,z)=\lim_{V\to\infty}\int_U^V\xi\int_0^\infty A(y)e^{-iy\xi^2/2}\sin(\ell\xi K_z(y)+a/\xi)\,dy\,d\xi
\end{split}
\label{eq:ordered-high-frequency-tail}
\end{align}
exists and satisfies
\begin{align*}
|P_U(a,\ell,z)|\leq C,
\end{align*}
uniformly in $a,\ell,z,U$.
\end{lemma}
\begin{proof}
Recall the property of $K_z$ from Proposition~\ref{prop:phase-prop} and \eqref{eq:Kz-property}:
\begin{align*}
0<K'_z(y)\leq K'_z(0),\quad K'_z(0)=\frac{2-z}{6},\quad \frac{1}{6}\leq K'_z(0)\leq\frac{1}{2}.
\end{align*}
By \eqref{eq:Kz-definition}, $K_z(0)=1+z\geq0$. By \eqref{eq:ps-asym} and direct differentiation of \eqref{eq:Kz-definition},
\begin{align}
A,A',A''\in L^1(0,\infty),\quad A(0)=1,\quad \sup_{\substack{-1\leq z\leq1\\y\geq0}}\big(|K''_z(y)|+|K'''_z(y)|\big)\leq C.
\label{eq:ordered-tail-regularity}
\end{align}
Set the following quantities:
\begin{align}
\begin{split}
c&=\ell K'_z(0),\quad S(\xi)=\ell K_z(0)\xi+\frac{a}{\xi},\\
d_-(y,\xi)&=\frac{\xi}{2}-\ell K'_z(y),\quad d_+(y,\xi)=\frac{\xi}{2}+\ell K'_z(y).
\end{split}
\label{eq:ordered-tail-notation}
\end{align}
Since $K'_z(y)\leq K'_z(0)$, $K'_z(0)\geq1/6$, and $\xi\geq U\geq4c$,
\begin{align}
\begin{split}
d_-(y,\xi) \geq \frac{\xi}{2}-c\geq\frac{\xi}{4},\quad d_+(y,\xi)\geq\frac{\xi}{2}, \quad
\ell \leq\frac{3}{2}U\leq\frac{3}{2}\xi.
\end{split}
\label{eq:ordered-tail-denominators}
\end{align}
The sine in \eqref{eq:ordered-high-frequency-tail} satisfies
\begin{align*}
\begin{split}
e^{-iy\xi^2/2}\sin(\ell\xi K_z(y)+a/\xi)=\frac{1}{2i}\big(e^{i(-y\xi^2/2+\ell\xi K_z(y)+a/\xi)}-e^{i(-y\xi^2/2-\ell\xi K_z(y)-a/\xi)}\big).
\end{split}
\end{align*}
By \eqref{eq:ordered-tail-notation},
\begin{align*}
\partial_y e^{i(-y\xi^2/2+\ell\xi K_z(y)+a/\xi)}&=-i\xi d_-(y,\xi)e^{i(-y\xi^2/2+\ell\xi K_z(y)+a/\xi)},\\
\partial_y e^{i(-y\xi^2/2-\ell\xi K_z(y)-a/\xi)}&=-i\xi d_+(y,\xi)e^{i(-y\xi^2/2-\ell\xi K_z(y)-a/\xi)}.
\end{align*}
One integration by parts in $y$, using \eqref{eq:ordered-tail-regularity} and \eqref{eq:ordered-tail-denominators}, gives
\begin{align}
\begin{split}
&\xi\int_0^\infty A(y)e^{-iy\xi^2/2}\sin(\ell\xi K_z(y)+a/\xi)\,dy\\
&=B(\xi)-\frac{1}{2}\int_0^\infty\big(\frac{A(y)}{d_-(y,\xi)}\big)'e^{i(-y\xi^2/2+\ell\xi K_z(y)+a/\xi)}\,dy\\
&\quad+\frac{1}{2}\int_0^\infty\big(\frac{A(y)}{d_+(y,\xi)}\big)'e^{i(-y\xi^2/2-\ell\xi K_z(y)-a/\xi)}\,dy,
\end{split}
\label{eq:ordered-tail-first-ibp}
\end{align}
where the boundary terms at $y=\infty$ vanish and the boundary term at $y=0$ is
\begin{align}
B(\xi) =\frac{1}{2}\big(\frac{e^{-iS(\xi)}}{\xi/2+c}-\frac{e^{iS(\xi)}}{\xi/2-c}\big) =-\frac{4c\cos S(\xi)}{\xi^2-4c^2}-i\big(\frac{1}{\xi-2c}+\frac{1}{\xi+2c}\big)\sin S(\xi).
\label{eq:ordered-tail-boundary}
\end{align}
We integrate each of the two integrals on the right-hand side of \eqref{eq:ordered-tail-first-ibp} once more in $y$. The integral containing $d_-$ satisfies
\begin{align}
\begin{split}
&\int_0^\infty\big(\frac{A}{d_-}\big)'e^{i(-y\xi^2/2+\ell\xi K_z(y)+a/\xi)}\,dy\\
&=-\frac{i e^{iS(\xi)}}{\xi d_-(0,\xi)}\big(\frac{A}{d_-}\big)'(0)-\frac{i}{\xi}\int_0^\infty\big(\frac{1}{d_-}\big(\frac{A}{d_-}\big)'\big)'e^{i(-y\xi^2/2+\ell\xi K_z(y)+a/\xi)}\,dy.
\end{split}
\label{eq:ordered-tail-second-ibp-minus}
\end{align}
The integral containing $d_+$ satisfies
\begin{align}
\begin{split}
&\int_0^\infty\big(\frac{A}{d_+}\big)'e^{i(-y\xi^2/2-\ell\xi K_z(y)-a/\xi)}\,dy\\
&=-\frac{i e^{-iS(\xi)}}{\xi d_+(0,\xi)}\big(\frac{A}{d_+}\big)'(0)-\frac{i}{\xi}\int_0^\infty\big(\frac{1}{d_+}\big(\frac{A}{d_+}\big)'\big)'e^{i(-y\xi^2/2-\ell\xi K_z(y)-a/\xi)}\,dy.
\end{split}
\label{eq:ordered-tail-second-ibp-plus}
\end{align}
The boundary terms at $y=\infty$ again vanish by \eqref{eq:ordered-tail-regularity}. 
We compute
\begin{align*}
\begin{split}
\big(\frac{A}{d_\pm}\big)'&=\frac{A'}{d_\pm}-\frac{A d'_\pm}{d_\pm^2},\\
\big(\frac{1}{d_\pm}\big(\frac{A}{d_\pm}\big)'\big)'&=\frac{A''}{d_\pm^2}-\frac{3A'd'_\pm}{d_\pm^3}-\frac{A d''_\pm}{d_\pm^3}+\frac{3A(d'_\pm)^2}{d_\pm^4}.
\end{split}
\end{align*}
Here the derivatives satisfy
\begin{align*}
d'_\pm(y,\xi)=\pm\ell K''_z(y),\quad d''_\pm(y,\xi)=\pm\ell K'''_z(y).
\end{align*}
Equations \eqref{eq:ordered-tail-regularity} and \eqref{eq:ordered-tail-denominators} imply
\begin{align}
\big|\big(\frac{A}{d_\pm}\big)'(0)\big| \leq\frac{C}{\xi},
\quad \int_0^\infty\big|\big(\frac{1}{d_\pm}\big(\frac{A}{d_\pm}\big)'\big)'\big|\,dy \leq\frac{C}{\xi^2}.
\label{eq:ordered-tail-remainder-estimates}
\end{align}
Substituting \eqref{eq:ordered-tail-remainder-estimates} into \eqref{eq:ordered-tail-second-ibp-minus} and \eqref{eq:ordered-tail-second-ibp-plus}, and then into \eqref{eq:ordered-tail-first-ibp}, gives
\begin{align}
\begin{split}
\xi\int_0^\infty A(y)e^{-iy\xi^2/2}\sin(\ell\xi K_z(y)+a/\xi)\,dy=B(\xi)+O(\xi^{-3})\quad (\xi\geq\max(U,1)),
\end{split}
\label{eq:ordered-tail-expansion}
\end{align}
where the implicit constant is independent of $a,\ell,z,U$.
The contribution from the first term on the right hand side of \eqref{eq:ordered-tail-boundary} is finite since the lower-bound $U\geq4c$.
By Lemma~\ref{lem:uniform-oscillatory-tails}, applied with
\begin{align*}
A=\ell K_z(0),\quad B=a,\quad C=2c,
\end{align*}
the integral
\begin{align*}
\int_U^\infty\frac{\sin S(\xi)}{\xi-2c}\,d\xi=G(\ell K_z(0),a,2c;U)
\end{align*}
is uniformly bounded, since $K_z(0)=1+z \geq 0$ and $U\geq4c=2(2c)$. For the term with denominator $\xi+2c$ in \eqref{eq:ordered-tail-boundary}, we have
\begin{align}
\begin{split}
\int_U^\infty\frac{\sin S(\xi)}{\xi+2c}\,d\xi&=E(\ell K_z(0),a;U)-\int_U^\infty\frac{2c\sin S(\xi)}{\xi(\xi+2c)}\,d\xi.
\end{split}
\label{eq:ordered-tail-positive-denominator}
\end{align}
The first term on the right-hand side of \eqref{eq:ordered-tail-positive-denominator} is uniformly bounded by Lemma~\ref{lem:uniform-oscillatory-tails}. For the second term, $\xi+2c\ge\xi$ and hence
\begin{align*}
\left|\int_U^\infty\frac{2c\sin S(\xi)}{\xi(\xi+2c)}\,d\xi\right|
\le 2c\int_U^\infty\frac{d\xi}{\xi^2}
\le\frac{2c}{U}\le\frac12,
\end{align*}
because $U\ge4c$. Therefore, we have
\begin{align}
\big|\int_U^\infty B(\xi)\,d\xi\big|\leq C.
\label{eq:ordered-tail-boundary-integral}
\end{align}
If $U<1$, then
\begin{align}
\begin{split}
\int_U^1\big|\xi\int_0^\infty A(y)e^{-iy\xi^2/2}\sin(\ell\xi K_z(y)+a/\xi)\,dy\big|\,d\xi&\leq\|A\|_{L^1}\int_U^1\xi\,d\xi\leq\frac{1}{2}\|A\|_{L^1}.
\end{split}
\label{eq:ordered-tail-compact-frequency}
\end{align}
If $U\geq1$, equations \eqref{eq:ordered-tail-expansion} and
\eqref{eq:ordered-tail-boundary-integral} prove the existence of
\eqref{eq:ordered-high-frequency-tail} and the uniform bound
\begin{align*}
|P_U(a,\ell,z)|\leq C.
\end{align*}
If $U<1$, apply \eqref{eq:ordered-tail-boundary-integral} with $U$ replaced by $1$.
The remainder in \eqref{eq:ordered-tail-expansion} is absolutely integrable for $\xi\geq1$,
while \eqref{eq:ordered-tail-compact-frequency} controls the integral from $U$ to $1$.
This proves the same conclusion when $U<1$.
\end{proof}


\begin{proposition}
\label{prop:difference-est}
For every $a\ge0$ and every $0\le \mathrm{R}\le \mathrm{L}$,
\begin{align*}
|\mathfrak D(a,\mathrm{L},\mathrm{R})|\leq Ca,
\end{align*}
where $C$ is an absolute constant independent of $a$, $\mathrm{L}$, and $\mathrm{R}$.
\end{proposition}

\begin{proof}
If $\mathrm{L}=0$, then $\mathrm{R}=0$ and the estimate follows from \eqref{eq:mkD-est'}. Assume $\mathrm{L}>0$. The differentiation formula \eqref{eq:direct-difference-derivative} and the decomposition \eqref{eq:Phat-decomposition} imply
\begin{align*}
\partial_a\mathfrak D(a,\mathrm{L},\mathrm{R})
=-\frac1{4\pi}\int_0^{2\pi}\widehat P_\theta(a)\,d\theta.
\end{align*}
The first term in \eqref{eq:Phat-decomposition} is bounded after angular integration by Lemma~\ref{lem:ps}, while the second term is bounded by \eqref{eq:Phat-second-bound}, which follows from Lemma~\ref{lem:ordered-high-frequency-tail}. Consequently,
\begin{align*}
\sup_{a\ge0}\sup_{0\le \mathrm{R}\le \mathrm{L}}|\partial_a\mathfrak D(a,\mathrm{L},\mathrm{R})|\le C.
\end{align*}
Since the two cosine terms in \eqref{eq:direct-difference} coincide when $a=0$, one has $\mathfrak D(0,\mathrm{L},\mathrm{R})=0$. The fundamental theorem of calculus therefore gives
\begin{align*}
|\mathfrak D(a,\mathrm{L},\mathrm{R})|
=\left|\int_0^a\partial_\alpha\mathfrak D(\alpha,\mathrm{L},\mathrm{R})\,d\alpha\right|
\le Ca.
\end{align*}
This proves the proposition.
\end{proof}

The preceding proposition closes the proof of the upper bound. Indeed, for $t>0$, \eqref{eq:positive-time-reduction} gives
\begin{align*}
\bigl|(e^{-itH_Z}P_{\mathrm{ac}}(H_Z))(x,y)\bigr|
\le C t^{-3/2}(1+a)
\le C\bigl(t^{-3/2}+Zt^{-1}\bigr),
\qquad a=Z\sqrt{2t}.
\end{align*}
Taking the supremum in $x,y$ proves \eqref{est:dis} for positive time. The negative-time estimate follows by taking adjoints, since $H_Z$ and $P_{\mathrm{ac}}(H_Z)$ are self-adjoint.

\section{Threshold lower bound and long-time sharpness}\label{sec:lower}

The upper bound in Theorem~\ref{thm:main} follows from the direct-difference estimate in Section~\ref{sec:est-Schrodinger-propagator}. We now prove that the inverse-time term is also necessary when $Z^2|t|$ is large. The argument uses the zero-energy behavior of the radial Coulomb spectral measure and the scaling of the full-Laplacian Hamiltonian.

\begin{lemma}
\label{lem:radial-spectral}
Let $Z>0$ and define
\begin{align*}
L_Z=-\frac{d^2}{dr^2}-\frac{Z}{r}
\end{align*}
on $L^2((0,\infty),dr)$ with the regular (Dirichlet) boundary condition at $r=0$. Let $\phi_Z(k,r)$ be the solution of
\begin{align*}
-\phi_Z''(r)-\frac{Z}{r}\phi_Z(r)=k^2\phi_Z(r),
\qquad \phi_Z(k,0)=0,\quad \partial_r\phi_Z(k,0)=1.
\end{align*}
Then the absolutely continuous spectral measure in the $k$-variable is
\begin{align*}
d\rho_Z(k)=\rho_Z'(k)\,dk,
\qquad
\rho_Z'(k)=\frac{2Zk}{1-e^{-\pi Z/k}},\qquad k>0.
\end{align*}
Moreover,
\begin{align*}
\phi_Z(0,r)=\sqrt{\frac rZ}\,J_1\!\bigl(\sqrt{4Zr}\bigr),
\end{align*}
and, for every $R_0<\infty$,
\begin{align*}
\sup_{0\le r\le R_0}|\phi_Z(k,r)-\phi_Z(0,r)|\le C_{Z,R_0}k^2
\qquad (0<k\le1).
\end{align*}
 More precisely, on every bounded $r$-interval the regular solution is smooth (in fact analytic)
 as a function of the parameter $k^2$ near $k=0$.
\end{lemma}

\begin{remark}
The operator $L_Z$ is unitarily equivalent to the radial $\ell=0$ part of $H_Z$ on $L^2((0,\infty),r^2dr)$ through
\begin{align*}
U:L^2((0,\infty),r^2dr)\longrightarrow L^2((0,\infty),dr),
\qquad (Uf)(r)=rf(r).
\end{align*}
Indeed, $U H_Z U^{-1}=L_Z$ on regular radial functions.
\end{remark}

\begin{proof}
Put $\eta=-Z/(2k)$ and let $F_0(\eta,x)$ denote the regular Coulomb wave. It satisfies
\begin{align*}
F_0(\eta,x)=C_0(\eta)x+O(x^2),
\qquad C_0(\eta)=e^{-\pi\eta/2}|\Gamma(1+i\eta)|,
\end{align*}
at the origin, and has unit sine amplitude at infinity. Hence
\begin{align*}
\phi_Z(k,r)=\frac{F_0(-Z/(2k),kr)}{C_0(-Z/(2k))k}.
\end{align*}
The large-$r$ asymptotic phase is
\begin{align*}
kr+\frac{Z}{2k}\log(2kr)+\sigma_0\!\left(-\frac{Z}{2k}\right),
\qquad \sigma_0(\eta)=\arg\Gamma(1+i\eta).
\end{align*}
The Wronskian identity for two regular solutions, followed by the standard delta-sequence argument at infinity, gives
\begin{align*}
\int_0^\infty\phi_Z(k,r)\phi_Z(k',r)\,dr
=\frac{\pi}{2k^2C_0(-Z/(2k))^2}\,\delta(k-k').
\end{align*}
Thus $\rho_Z'(k)=\frac2\pi k^2C_0(-Z/(2k))^2$. Since
\begin{align*}
C_0(\eta)^2=\frac{2\pi\eta}{e^{2\pi\eta}-1},
\end{align*}
substitution of $\eta=-Z/(2k)$ yields
\begin{align*}
\rho_Z'(k)=\frac{2Zk}{1-e^{-\pi Z/k}}.
\end{align*}
Completeness of this generalized eigenfunction expansion follows from the Hankel--Whittaker diagonalization of the radial Coulomb operator; equivalently, it is the radial restriction of Proposition~\ref{prop:spectral-decomp}.

For the zero-energy solution, set $s=\sqrt{4Zr}$ and
\begin{align*}
\phi_0(r)=\sqrt{\frac rZ}J_1(s)=\frac{s}{2Z}J_1(s).
\end{align*}
The Bessel equation and $ds/dr=2Z/s$ give
\begin{align*}
-\phi_0''(r)-\frac Zr\phi_0(r)=0.
\end{align*}
Since $J_1(s)=s/2+O(s^3)$, one has $\phi_0(r)=r+O(r^2)$, so $\phi_0$ is the regular zero-energy solution.

The regular solution satisfies the Volterra equation
\begin{align*}
\phi_Z(k,r)=r-\int_0^r(r-s)\left(\frac Zs+k^2\right)\phi_Z(k,s)\,ds.
\end{align*}
On every bounded interval, the Volterra series is uniformly convergent in $k^2$, and $\phi_Z(k,s)=O(s)$ removes the apparent singularity at $s=0$. Subtracting the equation at $k=0$ and iterating the resulting Volterra inequality gives the stated $O(k^2)$ bound.
\end{proof}

\begin{proposition}
\label{prop:threshold-lower}
Let $H_Z=-\Delta-Z|x|^{-1}$ on $\mathbb R^3$. There are universal constants $c>0$ and $T_0>0$ such that
\begin{align*}
Z^2|t|\ge T_0
\quad\Longrightarrow\quad
\bigl\|e^{-itH_Z}P_{\mathrm{ac}}(H_Z)\bigr\|_{L^1\to L^\infty}
\ge c\,\frac Z{|t|}.
\end{align*}
\end{proposition}

\begin{proof}
We first consider $Z=1$ and $t>0$. Choose $r_0>0$ with $\phi_0(r_0)\ne0$, and choose a nonnegative smooth radial function $f$ supported in an annulus on which $\phi_0$ has fixed nonzero sign. Set
\begin{align*}
A_0=\int_0^\infty\phi_0(s)s f(s)\,ds\ne0.
\end{align*}
The radial spectral resolution gives
\begin{align*}
E(t,r_0)=\frac1{r_0}\int_0^\infty e^{-itk^2}\phi_1(k,r_0)A_f(k)\rho_1'(k)\,dk,
\qquad
A_f(k)=\int_0^\infty\phi_1(k,s)s f(s)\,ds.
\end{align*}
Choose a smooth cutoff $\chi$ supported near zero and equal to one there. With $u=k^2$,
\begin{align*}
\frac1{r_0}\int_0^\infty e^{-itk^2}\chi(k^2)\phi_1(k,r_0)A_f(k)\rho_1'(k)\,dk
=\frac1{r_0}\int_0^\infty e^{-itu}g(u)\,du,
\end{align*}
where
\begin{align*}
g(u)=\frac{\chi(u)\phi_1(\sqrt u,r_0)A_f(\sqrt u)}{1-e^{-\pi/\sqrt u}},
\end{align*}
with $e^{-\pi/\sqrt u}$ defined to be zero at $u=0$. Since this factor is flat at $u=0$ and $\chi$ is supported in a sufficiently small neighborhood of zero, $g\in C_c^2([0,\infty))$ and
\begin{align*}
g(0)=\phi_0(r_0)A_0.
\end{align*}
Integrating by parts twice, with the boundary terms at infinity vanishing and the endpoint at $u=0$ retained, yields
\begin{align*}
\int_0^\infty e^{-itu}g(u)\,du=-\frac{i g(0)}{t}+O(t^{-2}).
\end{align*}

For the complementary part, which is supported in $k\ge k_0>0$, the exact Volterra equation is
\begin{align*}
\phi_1(k,r)=\frac{\sin(kr)}{k}-\frac1k\int_0^r\sin(k(r-s))\frac{\phi_1(k,s)}s\,ds.
\end{align*}
Repeated differentiation in $k$ gives polynomial bounds on compact $r$-intervals. Since $sf(s)$ is smooth and compactly supported away from zero, repeated Green's identities imply rapid decay of $A_f(k)$ and all of its derivatives. Because $k\ge k_0$ on this part, repeated integration by parts in $k$, using $\partial_k e^{-itk^2}=-2itk e^{-itk^2}$, gives an $O(t^{-2})$ contribution. Consequently,
\begin{align*}
E(t,r_0)=-\frac{i\phi_0(r_0)A_0}{r_0t}+O(t^{-2}).
\end{align*}
The leading coefficient is nonzero. Hence, for all sufficiently large $t$,
\begin{align*}
|E(t,r_0)|\ge c_0t^{-1}.
\end{align*}
Continuity in $r$ gives the same lower bound on a neighborhood of $r_0$, so the essential supremum has the same lower bound. Dividing by the fixed $L^1$ norm of $f$ proves
\begin{align*}
\bigl\|e^{-itH_1}P_{\mathrm{ac}}(H_1)\bigr\|_{L^1\to L^\infty}\ge c t^{-1}.
\end{align*}
The case $t<0$ follows by taking the adjoint, so the same lower bound holds with $t$ replaced by $|t|$.

For general $Z$, let $(S_Zg)(x)=Z^{3/2}g(Zx)$. Then
\begin{align*}
H_ZS_Z=Z^2S_ZH_1,
\qquad
K_Z(t;x,y)=Z^3K_1(Z^2t;Zx,Zy),
\end{align*}
where $K_Z$ is the absolutely continuous propagator kernel. Therefore
\begin{align*}
\bigl\|e^{-itH_Z}P_{\mathrm{ac}}(H_Z)\bigr\|_{L^1\to L^\infty}
=Z^3\bigl\|e^{-iZ^2tH_1}P_{\mathrm{ac}}(H_1)\bigr\|_{L^1\to L^\infty}
\ge c\frac Z{|t|}
\end{align*}
whenever $Z^2|t|\ge T_0$.
\end{proof}

\begin{proposition}[Long-time sharpness]\label{cor:sharp-full}
There are universal constants $c,C>0$ and $T_1>0$ such that
\begin{align*}
Z^2|t|\ge T_1
\quad\Longrightarrow\quad
c\frac Z{|t|}
\le
\bigl\|e^{-itH_Z}P_{\mathrm{ac}}(H_Z)\bigr\|_{L^1\to L^\infty}
\le C\frac Z{|t|}.
\end{align*}
In particular, the upper bound in Theorem~\ref{thm:main} is sharp in the long-time regime.
\end{proposition}

\begin{proof}
The lower bound is Proposition~\ref{prop:threshold-lower}. For the upper bound, Theorem~\ref{thm:main} gives
\begin{align*}
\bigl\|e^{-itH_Z}P_{\mathrm{ac}}(H_Z)\bigr\|_{L^1\to L^\infty}
\le C\bigl(|t|^{-3/2}+Z|t|^{-1}\bigr).
\end{align*}
When $Z^2|t|\ge1$, one has $|t|^{-3/2}\le Z|t|^{-1}$, which proves the claim after increasing $T_1$ if necessary.
\end{proof}

\begin{center}

\end{center}

\end{document}